\documentclass[12pt]{amsart}
\usepackage{amssymb}
\usepackage{amsbsy}
\usepackage{amscd}

\usepackage{mathrsfs}
\usepackage{nicefrac}

\usepackage{epic}
\usepackage{eepic}
\usepackage{youngtab}

\makeatletter
\renewcommand{\thesubsection}{\thesection(\@roman\c@subsection)}
\makeatother
\usepackage{verbatim}
\usepackage{version}
\usepackage{color}
\newenvironment{NB}{
\color{red}{\bf NB}. \footnotesize
}{}

\excludeversion{NB}
\excludeversion{NB2}
\newtheorem{Theorem}[equation]{Theorem}
\newtheorem{Corollary}[equation]{Corollary}
\newtheorem{Lemma}[equation]{Lemma}
\newtheorem{Proposition}[equation]{Proposition}

\theoremstyle{definition}

\theoremstyle{remark}
\newtheorem{Remark}[equation]{Remark}

\numberwithin{equation}{section}

\newcommand{\thmref}[1]{Theorem~\ref{#1}}
\newcommand{\secref}[1]{\S\ref{#1}}
\newcommand{\lemref}[1]{Lemma~\ref{#1}}
\newcommand{\propref}[1]{Proposition~\ref{#1}}
\newcommand{\corref}[1]{Corollary~\ref{#1}}
\newcommand{\subsecref}[1]{\S\ref{#1}}

\newcommand{\defeq}{\overset{\operatorname{\scriptstyle def.}}{=}}
\newcommand{\CC}{{\mathbb C}}
\newcommand{\ZZ}{{\mathbb Z}}
\newcommand{\QQ}{{\mathbb Q}}

\newcommand{\proj}{{\mathbb P}}
\newcommand{\CP}{\proj}

\newcommand{\GL}{\operatorname{GL}}

\newcommand{\End}{\operatorname{End}}
\newcommand{\Hom}{\operatorname{Hom}}
\newcommand{\Ext}{\operatorname{Ext}}
\newcommand{\Ker}{\operatorname{Ker}}

\newcommand{\Ima}{\operatorname{Im}}

\newcommand{\tr}{\operatorname{tr}}

\newcommand{\ve}{\varepsilon}
\newcommand{\Wedge}{{\textstyle \bigwedge}}
\newcommand{\shfO}{\mathcal O}

\renewcommand{\MR}[1]{}
\newcommand{\linf}{\ell_\infty}

\newcommand{\Hilb}[2][X]{{#1}^{[#2]}}
\newcommand{\ch}{\operatorname{ch}}
\newcommand{\loc}{\mathbb F}
\newcommand{\ra}{\rangle}
\newcommand{\la}{\langle}
\newcommand{\bk}{\boldsymbol k}
\newcommand{\idl}{I}
\newcommand{\Supp}{\operatorname{Supp}}
\newcommand{\normal}[1]{\mbox{\large\bf:}#1\mbox{\large\bf:}}
\newcommand{\PP}[2]{P_{#2}({#1})}

\newcommand{\hf}{\hfil}
\newcommand{\hs}{\heartsuit}
\newcommand{\sps}{\spadesuit}
\newcommand{\cs}{\clubsuit}
\newcommand{\ds}{\diamondsuit}

\usepackage{url}

\makeindex

\usepackage[
bookmarks=true,
colorlinks=true]{hyperref}

\begin{document}
\title[More lectures on Hilbert schemes]
{More lectures on Hilbert schemes of points on surfaces
}
\author{Hiraku Nakajima}
\address{Research Institute for Mathematical Sciences,
Kyoto University, Kyoto 606-8502,
Japan}
\email{nakajima@kurims.kyoto-u.ac.jp}

\dedicatory{Dedicated to Professor Shigeru Mukai on the occasion of
  his 60th birthday}

\maketitle

\section*{Introduction}

This paper is based on author's lectures at Kyoto University in 2010
Summer, and in the 6th MSJ-SI `Development of Moduli Theory' at RIMS
in June 2013.

The purpose of lectures was to review several results on Hilbert
schemes of points which were obtained after author's lecture note
\cite{Lecture} was written. Among many results, we choose those which
are about equivariant homology groups $H^T_*(\Hilb{n})$ of Hilbert
schemes of points on the affine plane $X = \CC^2$ with respect to the
torus action. Study of equivariant homology groups increases its
importance recently. In particular, it is a basis of the AGT
correspondence between instanton moduli spaces on $\CC^2$ and the
representation theory of $W$-algebras, which is a very hot topic now
(see e.g., \cite{MO}).

We omit proofs if they are present in \cite{Lecture}, but give
self-contained proofs otherwise. In this sense, this should be read
after \cite{Lecture}.

The paper is organized as follows. In \secref{sec:equiv}, we review
basics on equivariant (co)homology groups. It will be basis of
subsequent sections. In \secref{sec:Hilb} we construct the Fock
representation of the Heisenberg algebra on $\bigoplus
H^T_*(\Hilb{n})$, following \cite[Ch.~8]{Lecture} as well as an idea
of Vasserot \cite{VasserotC2}. In \secref{sec:Jack} we explain a
geometric realization of Jack symmetric functions as fixed point
classes in $H^T_*(\Hilb{n})$ by Li-Qin-Wang \cite{LQW-jack}. We also
give author's unpublished result, which was used in
\cite{LQW-jack}. As applications, we give geometric proofs of the norm
formula and Pieri formula of Jack symmetric functions. In
\secref{sec:Vir} we construct a representation of the Virasoro algebra
on $\bigoplus H^T_*(\Hilb{n})$. It is a special case of Lehn's result
\cite{Lehn} for $X=\CC^2$, but the proof is different.

\subsection*{Acknowledgment}
This work was supported by the Grant-in-Aid for Scientific Research
(B) (No.~23340005), Japan Society for the Promotion of Science.

\section{Equivariant (co)homology groups}\label{sec:equiv}

In this section, we review basics on equivariant cohomology and
homology groups, which will be used in later sections. Our definition
of equivariant cohomology groups is slightly different from the usual
one (e.g., as in \cite{Audin}). We replace the classifying space $BT$
by its finite dimensional approximations. This approach is suitable
for equivariant homology groups, and was taken by Lusztig
\cite{Lu-cus}. The same approach was used for equivariant Chow groups
\cite{Edidin-Graham}.

We assume coefficients of cohomology groups are $\CC$.

\subsection{Equivariant cohomology groups}

Let $T=(\CC^*)^r$ be an algebraic torus. Let $V = (\CC^{N+1})^r$ be a
$T$-module, where each $\CC^*$ acts on the corresponding $\CC^N$ by
multiplication. Then $V\setminus\{0\}\to (\CP^{N})^r$ is a $T$-bundle,
and the universal $T$-bundle $ET\to BT$ is the inductive limit when
$N\to \infty$.
Let $M$ be a $T$-variety, i.e., an algebraic variety with an algebraic
$T$-action. We further assume that $M$ admits a locally closed
$T$-embedding into a smooth projective $T$-variety. This restriction
can be weakened, but it is enough for our purpose.

We define the equivariant cohomology of $M$ by
\begin{equation*}
  H^i_T(M) \defeq H^i(M_V)
  \quad\text{where }M_V = (V\setminus\{0\})\times_T M.
\end{equation*}
Here for a given $i$, we take $V$ with sufficiently large $N$. Then it is well-defined thanks to the following lemma:

\begin{Lemma}
  $H^i_T(M)$ is independent of the choice of $V$.
\end{Lemma}

\begin{proof}
  A key point is that $H^i(V\setminus\{0\}) \cong H^i( (S^{2N+1})^r) =
  0$ if $0 < i < 2N+1$. 

  Take two large vector spaces $V_1$, $V_2$ and consider the diagram
  \begin{equation*}
    M_{V_1} \leftarrow 
    (V_1\setminus\{0\}\times V_2\setminus\{0\})\times_T M
    \rightarrow M_{V_2}.
  \end{equation*}
  The left and right arrows are fiber bundles with fibers
  $V_2\setminus\{0\}$, $V_1\setminus\{0\}$ respectively. Since their
  cohomology groups vanishes in degree between $1$ and sufficiently
  large number (in particular, larger than $i$), $H^i(M_{V_1})$ and
  $H^i(M_{V_2})$ are isomorphic to $H^i$ of the middle under the
  pull-back homomorphisms.
\end{proof}

Let us briefly explain several important properties of equivariant
cohomology groups.

When $M = \mathrm{pt}$, we have $H^i_T(\mathrm{pt}) =
H^i((\CP^N)^r)$. Note that $H^*(\CP^N) = \CC[a]/(a^{N+1} = 0)$.
Taking $N\to\infty$, we have
\begin{equation}\label{eq:20}
  H^*_T(\mathrm{pt}) \cong \CC[a_1,\dots,a_r],
\end{equation}
where $a_i$ is the first Chern class of the hyperplane bundle
$\shfO(1)$ of the $i^{\mathrm{th}}$ factor of $(\CP^N)^r$.
As this example shows, $H^i_T(M)$ may be nonzero in arbitrary large
degree $i$ unlike ordinary cohomology groups.

When $T$ acts trivially on $M$, we have
\begin{equation}\label{eq:24}
    H^*_T(M) \cong H^*(M)\otimes H^*_T(\mathrm{pt})
\end{equation}
from the definition.

We have a cup product
\begin{equation*}
  H^i_T(M)\otimes H^j_T(M)\to H^{i+j}_T(M).
\end{equation*}
The isomorphism \eqref{eq:20} is a ring isomorphism.

If $f\colon M_1\to M_2$ is a $T$-equivariant continuous map, we have a
pull-back map $H^*_T(M_2)\to H^*_T(M_1)$. It is a ring
homomorphism. In particular, we always have $H^*_T(\mathrm{pt})\to
H^*_T(M)$. Therefore $H^*_T(M)$ is a ring over $H^*_T(\mathrm{pt})
\cong \CC[a_1,\dots,a_r]$. We consider $H^*_T(M)$ as a coherent sheaf
on $\operatorname{Spec}(\CC[a_1,\dots,a_r]) = \CC^r$ in this way, and
this view point is useful in the localization theorem below.

Suppose $T$ acts freely on $M$, and $M\to M/T$ is a principal
$T$-bundle. Then $M_V$ is a fiber bundle over $M/T$ with fiber
$V\setminus\{0\}$. Since the fiber has trivial cohomology groups, the
spectral sequence for a fiber bundle gives us
\begin{equation*}
  H^i_T(M) \cong H^i(M/T).
\end{equation*}
Moreover, $H^{>0}_T(\mathrm{pt})$ acts by $0$, as it is so on the $E^2$ term.
In this case, $H^i_T(M)$ vanishes if $i$ is sufficiently large for a
reasonable $M/T$.

Note that $M_V\to (\CP^N)^r$ is a fiber bundle with fiber $M$. The
restriction to a fiber gives a forgetful homomorphism
\begin{equation*}
  H^*_T(M)\to H^*(M).
\end{equation*}
More generally, we have a restriction homomorphism
\begin{equation*}
  H^*_T(M)\to H^*_{T'}(M)
\end{equation*}
for a subtorus $T'\subset T$. Considering the case $M=\mathrm{pt}$, we
find that we have an intrinsic description of \eqref{eq:20}:
\begin{equation*}
  H^*_T(\mathrm{pt}) \cong \CC[\operatorname{Lie}T],  
\end{equation*}
where $\CC[\operatorname{Lie}T]$ is the ring of polynomial functions
on $\operatorname{Lie}T$. The homomorphism $H^*_T(\mathrm{pt}) =
\CC[\operatorname{Lie}T]\to H^*_{T'}(\mathrm{pt}) =
\CC[\operatorname{Lie}T']$ is induced from the embedding
$\operatorname{Lie}T'\to \operatorname{Lie}T$.

If $E$ is a $T$-equivariant (complex) vector bundle over $M$, it
induces a vector bundle $E_V = (V\setminus\{0\})\times_T E$ over
$M_V$. We define an equivariant Chern class $c_i(E)$ by $c_i(E_V)$. 
If $E$ is rank $r$, the top Chern class $c_r(E)$ is equal to the
equivariant Euler class $e(E)$, defined in the same way.

If $M=\mathrm{pt}$, a $T$-equivariant vector bundle is nothing but a
representation of $T$.
The above $a_i$ in \eqref{eq:20} is the equivariant first Chern
class of the representation $T\to \CC^*$, the projection to the
$i^{\mathrm{th}}$ factor. Then $c_i(E)$ is the $i^{\mathrm{th}}$ elementary symmetric function of weights of $E$, regarded as a representation of $T$.
Here weights are considered as linear functions
$\operatorname{Lie}T\to\CC$.
The equivariant Euler class $e(E)$ is the product of weights.

\subsection{Equivariant homology groups}

In \cite[Ch.8]{Lecture} we used both ordinary and Borel-Moore (or
locally finite) homology groups to deal with Hilbert schemes of points
on a noncompact surface $X$, like $X=\CC^2$. Here in the equivariant
case, we mainly use Borel-Moore homology groups, since it fits better with convolution products.

Let $M$ be a $T$-variety as above. We define
\begin{equation*}
  H_i^{T,lf}(M) \defeq H^{lf}_{i+2\dim V - 2\dim T}(M_V).
\end{equation*}
It is independent of $V$ by the same argument as above, where we use
$H^{lf}_i(V\setminus\{0\}) = 0$ for $2(\dim V - N - 1)< i < 2\dim V$.
Then $H^{T,lf}_*(M)$ is a module over $H^*_T(M)$ under the cap
product. In particular, it is also a module over $H^*_T(\mathrm{pt})$.

If $T$ acts free on $M$ and $M\to M/T$ is a fiber bundle, we have
\begin{equation}\label{eq:21}
  H^{T,lf}_i(M) \cong H^{lf}_{i-2\dim T}(M/T).
\end{equation}

If $M$ is a smooth manifold of $\dim M = m$ with a smooth $T$-action,
we define the equivariant fundamental class $[M]$ as $[M_V]\in
H^{lf}_{m+2\dim V- 2\dim T}(M_V)$. We have the Poincar\'e duality
\begin{equation}\label{eq:PD}
  H^i_T(M) \cong H^{T,lf}_{m-i}(M); \qquad c\mapsto c\cap [M]
\end{equation}

Even for an irreducible complex algebraic variety $M$ with a
$T$-action, not necessarily smooth, its fundamental class $[M]\in
H^{T,lf}_{2\dim M}(M)$ is defined, as $[M_V]$ is defined. However the
homomorphism $H^i_T(M)\to H^{T,lf}_{2\dim M-i}(M)$ may not be an
isomorphism in general. Note also that
\begin{equation}\label{eq:23}
  H^{T,lf}_i(M) = 0 \quad \text{if $i>2\dim M$}.
\end{equation}
On the other hand, $H^{T,lf}_i(M)$ may be nonzero even for $i < 0$.

If $f\colon M_1\to M_2$ is a proper $T$-equivariant map, we have the
push-forward homomorphism $f_*\colon H^{T,lf}_*(M_1)\to H^{T,lf}_*(M_2)$.

Considering the spectral sequence for the fiber bundle $M_V\to
(\CP^N)^r$, we have a forgetful homomorphism
\begin{equation*}
  H^{T,lf}_i(M)\to H^{lf}_i(M).
\end{equation*}

We define the ordinary equivariant homology group $H^T_i(M)$ as the
dual space to $H^i_T(M)$. This is enough for our purpose thanks to the
universal coefficient theorem, as we only consider complex
coefficients.
\begin{NB}
  More concretely, $H^T_i(M) = H_i(M_V)$.
\end{NB}%
We have $f_*\colon H^T_i(M_1)\to H^T_i(M_2)$ for a $T$-equivariant
map, not necessarily proper. It is defined as the transpose of
$f^*$. If $M$ is a smooth $T$-manifold, we have the cap product
\begin{equation*}
  \cap \colon 
  H^{T,lf}_{i}(M)\otimes H^T_{j}(M)\to H^{T}_{i+j-\dim M}(M)
\end{equation*}
thanks to \eqref{eq:PD}. We define the intersection pairing
\begin{equation*}
  H^{T,lf}_{i}(M)\otimes H^T_{j}(M)\to H^{T}_{i+j-\dim M}(\mathrm{pt})
  \cong H^{\dim M-i-j}_T(\mathrm{pt})
\end{equation*}
as $a_{M*}(\bullet\cap\bullet)$, where $a_M\colon M\to
\mathrm{pt}$. The second isomorphism is the Poincar\'e duality for
$\mathrm{pt}$.

We do not review further properties of equivariant Borel-Moore
homology groups, which we will use implicitly in the next
section. They are listed in \cite[\S8.2]{Lecture} for nonequivariant
Borel-Moore homology groups, and equivariant versions are simple
consequences of nonequivariant ones applied to $M_V$.

\subsection{Localization theorem}

In this subsection we explain the localization theorem in equivariant
Borel-Moore homology groups, which relates the equivariant Borel-Moore
homology of $M$ and that of the fixed point set $M^T$. In many
situations, the fixed point set consists of a finite set, so the
latter is just a direct sum of the equivariant cohomology of points
(cf.\ \eqref{eq:24}). Therefore the localization theorem is useful to
say something on $H^{T,lf}_*(M)$.

A key point is that we view $H_*^T(M)$ as a module over
$H^*_T(\mathrm{pt}) \cong \CC[\operatorname{Lie}T]$, or a sheaf on
$\operatorname{Lie}T$.

Recall that $H^{>0}_T(\mathrm{pt})$ acts trivially on $H_*^T(M)$ when
$T$ acts freely on $M$, see \eqref{eq:21}. Therefore the support
of $H_*^T(M)$ is $0$ in $\operatorname{Lie}T$.
More generally, we have
\begin{Lemma}\label{lem:strata}
    Suppose that the stabilizers of arbitrary points $x\in M$ is a
    fixed subgroup of $T'\subset T$. Then the support of $H_*^T(M)$ is
    contained in $\operatorname{Lie}(T')$.
\end{Lemma}

\begin{proof}
    By the assumption, $M_V$ is a fiber bundle over $M/(T/T')$ with
    fiber $V\setminus\{0\}/T'$. Then the action of
    $H^*_T(\mathrm{pt})\cong \CC[\operatorname{Lie}T]$ factors through
    $\CC[\operatorname{Lie}T']$ on the $E^2$ term, and hence also on
    $H_*^T(M)$.
\end{proof}

Let $M^T$ be the fixed point set in $M$, and consider
$H^T_*(M\setminus M^T)$. Since $M$ can be equivariantly embedded into
a projective space, there are only finitely many stabilizers occur. We
claim
\begin{equation}\label{eq:22}
    \operatorname{Supp} H^T_*(M\setminus M^T)
    \subset \bigcup_{x\in M\setminus M^T}      \operatorname{Lie}(\operatorname{Stab}(x)).
\end{equation}
We decompose $M\setminus M^T$ according to stabilizers into $\bigsqcup
M_\alpha$. And we can order the index set $\{\alpha\}$ so that $M_{\le
  \alpha} = \bigcup_{\beta:\beta\le \alpha} M_\beta$ is closed in
$M\setminus M^T$. We set $M_{<\alpha} = \bigcup_{\beta:\beta <
  \alpha} M_\beta$. From an exact sequence
\begin{equation*}
    \cdots \to H^T_i(M_{< \alpha}) \to H^T_i(M_{\le \alpha}) \to H^T_i(M_\alpha) 
    \to H^T_{i-1}(M_{< \alpha}) \to \cdots
\end{equation*}
and \lemref{lem:strata} applied to $M_\alpha$, we deduce that
the support of
\(
    H^T_*(M_{\le \alpha})
\)
is contained in the right hand side of \eqref{eq:22} by an induction
on $\alpha$. Therefore using the exact sequence for
$H^T_*(M)$, $H^T_*(M^T)$, $H^T_*(M\setminus M^T)$, we get
\begin{Theorem}
    Let $i\colon M^T\to M$ be the inclusion of the fixed point
    set. Then the kernel and cokernel of the homomorphism
    \begin{equation*}
        i_*\colon H^T_*(M^T)\to H^T_*(M)
    \end{equation*}
    are supported in $\bigcup_{x\in M\setminus M^T}      \operatorname{Lie}(\operatorname{Stab}(x))$.
\end{Theorem}

\begin{Corollary}\label{cor:i*}
    Let $\operatorname{Frac}(H^*_T(\mathrm{pt}))$ be the fractional
    field of $H^*_T(\mathrm{pt})$ and let $H^T_*(M)_\loc =
    H^T_*(M)\otimes_{H^*_T(\mathrm{pt})}\operatorname{Frac}(H^*_T(\mathrm{pt}))$, and
    similarly for $H^T_*(M^T)_\loc$.
    Then
    \begin{equation*}
        i_*\colon H^T_*(M^T)_\loc\to H^T_*(M)_\loc
    \end{equation*}
    is an isomorphism.
\end{Corollary}

\subsection{Fixed point formula}\label{sec:fixed}

When $M$ is nonsingular, we have the Poincar\'e duality \eqref{eq:PD}. Therefore we have two isomorphisms
\(
   i_*\colon H^T_*(M^T)_\loc \to H^T_*(M)_\loc
\)
and
\(
   i^*\colon H^T_*(M)_\loc \to H^T_*(M^T)_\loc.
\)
Having both are are very useful, as $i^*i_*$ can be explicitly
written down. This leads us to the {\it fixed point formula}.

Let $M^T = \bigsqcup M_\alpha$ be the decomposition to connected
components. Each $M_\alpha$ is a nonsingular subvariety of $M$. Let
$N_\alpha$ denote its normal bundle.  We have
\begin{equation*}
    H^T_*(M^T) = \bigoplus_\alpha H^T_*(M_\alpha).
\end{equation*}

\begin{Lemma}\label{lem:1}
    Let $i_\alpha$ be the inclusion of $M_\alpha$ into $M$. Then
    $i_\alpha^* i_{\alpha*}$ is given by the cap product
    $e(N_\alpha)\cap\bullet$ of the equivariant Euler class of $N_\alpha$.
\end{Lemma}

\begin{proof}
    In the neighborhood of $M_\alpha$, $M$ is isomorphic to a
    neighborhood of the $0$-section of $N_\alpha$. Therefore we may
    replace $i_\alpha\colon M_\alpha\to M$ by the inclusion
    $M_\alpha\to N_\alpha$ of the $0$-section. We then obtain the
    assertion by the Thom isomorphism.
\end{proof}

\begin{Lemma}\label{lem:2}
    $e(N_\alpha)\cap\bullet$ is invertible in $H^T_*(M_\alpha)_\loc$.
\end{Lemma}

\begin{proof}
    Since $T$ acts trivially on $M_\alpha$, we have $H^*_T(M_\alpha) =
    H^*(M_\alpha)\otimes H^*_T(\mathrm{pt})$. Note that
    $H^{>0}(M_\alpha)\otimes H^*_T(\mathrm{pt})$ is nilpotent as
    $H^{>2\dim M_\alpha}(M_\alpha) = 0$. Therefore it is enough to
    check that $H^0(M_\alpha)\otimes H^*_T(\mathrm{pt})$ part of
    $e(N_\alpha)$ is nonzero. It is enough to study it after
    restricting $N_\alpha$ to a point $x\in M_\alpha$.

    For $x\in M_\alpha$, $T_x M$ is a $T$-module so that $T_x
    M_\alpha$ is its weight $0$ subspace. Therefore
    $\left.e(N_\alpha)\right|_x$ is the product of nonzero weights of
    $T_x M_\alpha$. Therefore it is nonzero.
\end{proof}

Let us denote the inverse of $e(N_\alpha)\cap\bullet$ by $\frac1{e(N_\alpha)}$.

\begin{Theorem}\label{thm:inverse}
    Suppose $M$ is nonsingular. Then the inverse of $i_*$ in
    \corref{cor:i*} is given by
    \begin{equation*}
        \sum_\alpha \frac1{e(N_\alpha)} i_\alpha^*.
    \end{equation*}
\end{Theorem}

\begin{proof}
    Note that we have decomposition $H^T_*(M^T) = \bigoplus
    H^T_*(M_\alpha)$, and hence we have $i_* = \sum i_{\alpha*}$. By
    Lemmas~\ref{lem:1},~\ref{lem:2}, $e(N_\alpha)^{-1} i_\alpha^*
    i_{\alpha*}$ is the identity operator on $H^T_*(M_\alpha)_\loc$.
    Therefore we have the assertion.
\end{proof}

We now arrive at Atiyah-Bott-Berline-Vergne fixed point formula.

\begin{Theorem}\label{thm:fixed}
    Assume $M$ is proper and nonsingular. Let $a\colon M\to
    \mathrm{pt}$, $a_\alpha\colon M_\alpha\to\mathrm{pt}$. Then we
    have an equality in $H^T_*(\mathrm{pt})_\loc\cong
    \CC(a_1,\dots,a_r)$ for $\omega\in H^T_*(M)$:
    \begin{equation*}
        a_*(\omega) = \sum_\alpha a_{\alpha*}
        \left(\frac1{e(N_\alpha)} i_\alpha^*\omega\right).
    \end{equation*}
\end{Theorem}

\begin{proof}
    We have
    \begin{equation*}
        a_*(\omega) = a_* i_* i_*^{-1} (\omega)
        = \sum_\alpha a_* i_{\alpha*} \frac1{e(N_\alpha)} i_\alpha^*(\omega).
    \end{equation*}
Since $a_\alpha = a \circ i_\alpha$, we get the assertion.
\end{proof}

\section{Equivariant homology groups of 
Hilbert schemes of points}\label{sec:Hilb}

We explain the construction of the Fock space representation of the
Heisenberg algebra in \cite[Ch.~8]{Lecture} in 
equivariant homology groups in this section. It was first noticed by
Vasserot \cite{VasserotC2} that the arguments in \cite[Ch.~8]{Lecture}
work in the equivariant setting, and such a generalization is quite
useful.

\subsection{Heisenberg algebra}\label{sec:heisenberg-algebra}

Let $X=\CC^2$ with the linear coordinate system $(z,\xi)$. Two
dimensional torus $T$ acts on $X$ by $(t_1,t_2)\cdot (z,\xi) = (t_1z,
t_2 \xi)$.

Let $\Hilb{n}$ denote the Hilbert scheme of $n$ points in $X$. Let
$\pi\colon\Hilb{n}\to S^nX$ be the Hilbert-Chow morphism. We have
induced $T$-actions on $\Hilb{n}$ and $S^nX$ so that $\pi$ is
equivariant.

Recall (\cite[(8.9)]{Lecture}) that we considered a subvariety
$P[i]\subset\bigsqcup_n \Hilb{n}\times\Hilb{n-i}\times X$ defined by
\begin{equation}\label{eq:6}
  P[i] \defeq \big\{ (I_1,I_2, x) \mid
    I_1\subset I_2, \Supp(I_2/I_1) = \{ x\} \big\}
\end{equation}
for $i > 0$. Let us omit $\bigsqcup_n$ for brevity hereafter.

The projections $q_1\colon P[i]\to \Hilb{n}$, $q_2\colon P[i]\to
\Hilb{n-i}\times X$ are proper. Therefore convolution operators
\begin{equation}\label{eq:4}
  \begin{split}
  &H^{T,lf}_*(\Hilb{n-i}\times X)\to 
  H^{T,lf}_*(\Hilb{n}); \bullet\mapsto 
  q_{1*}(q_2^*(\bullet)\cap [P[i]]),
\\
  &H^{T,lf}_*(\Hilb{n})\to 
  H^{T,lf}_*(\Hilb{n-i}\times X); \bullet\mapsto 
  (-1)^i q_{2*}(q_1^*(\bullet)\cap [P[i]])
  \end{split}
\end{equation}
are well-defined. The sign $(-1)^i$ is introduced so that the
commutation relation below is simplified. We take the direct sum of
homology groups over $n$ later, keeping the same notation.

Take an equivariant class $\beta\in H^{T,lf}_*(X)$. We replace
$\bullet$ by $\bullet\otimes\beta\in H^{T,lf}_*(\Hilb{n-i}\times
X)\cong H^{T,lf}_*(\Hilb{n-i})\otimes H^{T,lf}_*(X)$ in the first
operator in \eqref{eq:4}. It is considered as an operator
$H^{T,lf}_*(\Hilb{n-i})\to H^{T,lf}_*(\Hilb{n})$. Let us denote it by
$\PP{\beta}{-i}$. It was denoted by $P_\beta[-i]$ in \cite{Lecture}.

On the other hand, for the second operator in \eqref{eq:4}, we take
the intersection pairing with $\alpha\in H^{T}_*(X)$ via
$H^{T,lf}_*(\Hilb{n-i}\times X)\cong H^{T,lf}_*(\Hilb{n-i})\otimes
H^{T,lf}_*(X)$. We obtain an operator $H^{T,lf}_*(\Hilb{n})\to
H^{T,lf}_*(\Hilb{n-i})$, which is denoted by $\PP{\alpha}{i}$.

These operators $\PP{\alpha}{i}$, $\PP{\beta}{-i}$ are the same as ones in
\cite[Ch.8]{Lecture} though the current explanation is slightly
different.

We can replace $H^T_*$ by the equivariant Borel-Moore homology group
$H^{T,lf}_*$ in \eqref{eq:4}. Then
\begin{equation}
    \begin{split}
      & \PP{\alpha}{i}\colon H^{T,lf}_*(\Hilb{n})\to H^{T,lf}_*(\Hilb{n-i}),\\
      & \PP{\beta}{-i}\colon H^{T,lf}_*(\Hilb{n-i})\to H^{T,lf}_*(\Hilb{n})
    \end{split}
\end{equation}
are well-defined for $\alpha\in H^T_*(X)$, $\beta\in
H^{T,lf}_*(X)$. Note that homology groups containing $\alpha$, $\beta$
are swapped from the above case.

We have a perfect pairing
\begin{equation}\label{eq:14}
  \la\ ,\ \ra\colon 
  H^{T,lf}_*(\Hilb{n})\otimes H^T_*(\Hilb{n}) \to H^T_*(\mathrm{pt});
  c\otimes c'\mapsto (-1)^n a_*(c\cap c')
\end{equation}
by the Poincar\'e duality. Here $a\colon \Hilb{n}\to \mathrm{pt}$. The
transpose of $\PP{\alpha}{i}$ is equal to $\PP{\alpha}{-i}$.

The operators are linear over $H^*_T(\mathrm{pt})$ from the
construction:
\(
   \PP{\alpha}{i} f = f \PP{\alpha}{i},
\)
etc.\ for $f\in H^*_T(\mathrm{pt})$. It is $H^*_T(\mathrm{pt})$-linear
on $\alpha$:
\(
   \PP{f\alpha}{i} = f\PP{\alpha}{i},
\)
etc.

Note that $H_*^T(X) \cong H^*_T(\mathrm{pt})[0]$, $H_*^{T,lf}(X) \cong
H^*_T(\mathrm{pt})[X]$. Therefore it is enough to consider the cases
$\alpha$, $\beta= [0]$ and $[X]$.

Then the following commutation relation of the Heisenberg algebra
holds:
\begin{equation}\label{eq:5}
  \left[ \PP{\alpha}{i}, \PP{\beta}{j}\right]
  = i\delta_{i+j,0} \la\alpha,\beta\ra\operatorname{id}.
\end{equation}
When the right hand side is nonzero, i.e., $i+j=0$, one of $\alpha$ or
$\beta$ is in $H^T_*(X)$ and the other is in
$H^{T,lf}_*(X)$. Therefore 
\begin{equation}\label{eq:18}
  \la \alpha,\beta\ra = - a_{X*}(\alpha\cap\beta) \in H^*_T(\mathrm{pt})
\end{equation}
is well-defined, where $a_X\colon X\to\mathrm{pt}$.
We take the same sign convention as above, understanding $\Hilb{1} =
X$, i.e., it is $(-1)$ times the intersection pairing.

The proof of \eqref{eq:5} for nonequivariant homology groups in
\cite[Ch.8]{Lecture} works also for the equivariant case.
Let us briefly explain checkpoints.
When $i+j\neq 0$, the right hand side of \eqref{eq:5} vanishes. In
these cases, we study set-theoretical intersections of cycles, check
that they are transversal intersections generically and non-generic
parts do not contribute to the computation by dimension reason. The
last reasoning works in the equivariant case thanks to \eqref{eq:23}.

When $i + j = 0$, the argument in \cite{Lecture} works before the
sentence ``If $\deg\alpha+\deg\beta < 4$, then $\alpha\cap\beta = 0$''
in the middle of p.102. This is not true for the equivariant case, as
$H^T_{<0}(X)$ is nonzero. So let us go back a little and see what is
actually proved there.
We consider
\begin{equation}
  L \defeq \left\{
    (\idl_1,\idl_3,x)\in \Hilb{n}\times\Hilb{n}\times X\, \middle|\,
    \text{$\idl_1 = \idl_3$ outside $x$}
    \right\}.
\end{equation}
Then $L$ has $n$ irreducible components $L_1$, \dots, $L_n$ with
$\dim_\CC L_i = 2n$, and $(2n+2)$-dimensional irreducible component
$\Delta_{\Hilb{n}}\times X$, and other irreducible components (if
exists) have lower dimension (see \cite[Lemma~8.32]{Lecture}). We have
the projection $\Pi''\colon L\to X$ to the third factor. Then by the
same argument as the nonequivariant case, we show that there is a
class $\mathscr R$ in $H_{4n+4}^{T,lf}(L)$ (which is $p_{134*}\iota''_*
R'' - p_{134*}\iota'''_* R'''$ in the notation \cite[p.101]{Lecture})
such that the left hand side of \eqref{eq:5} is given by the
correspondence 
\begin{equation}
  p_{13*}(\Pi^{\prime\prime*}(\alpha\cap\beta)\cap\mathscr R)
  \in H_{4n + \deg\alpha+\deg\beta-4}^{T,lf}(\Hilb{n}\times_{S^nX}\Hilb{n}),
\end{equation}
where $p_{13}\colon L\to \Hilb{n}\times_{S^nX}\Hilb{n}$ is the
projection $(\idl_1,\idl_3,x)\mapsto (\idl_1,\idl_3)$.
By dimension reason, we see that $\mathscr R$ comes from a class in
$H^{T,lf}_{4n+4}(\Delta_{\Hilb{n}}\times X)$. Since it is of top
degree, it must be $c_{i,n}[\Delta_{\Hilb{n}}\times X]$, where the
multiple constant $c_{i,n}$ is independent of equivariant variables, i.e.,
a complex number.
Therefore
\begin{equation}
  p_{13*}(\Pi^{\prime\prime*}(\alpha\cap\beta)\cap\mathscr R)
  = c_{i,n} \la \alpha,\beta\ra [\Delta_{\Hilb{n}}].
\end{equation}

Now the only task is to determine $c_{i,n}$. Since $c_{i,n}$ is a
complex number, we can take nonequivariant limit: it is the same for
equivariant and nonequivariant cases. Therefore we conclude $c_{i,n} =
i$.

\subsection{Torus fixed points}\label{sec:torus}

For a later purpose, let us recall fixed points with respect to the
two dimensional torus $T$. They are parametrized by partitions
$\lambda$ of $n$, or equivalently Young diagrams $D$ with $n$
boxes. (See \cite[Ch.~5 and 6]{Lecture}.) We keep our convention in
\cite{Lecture}: for $\lambda = (\lambda_1,\lambda_2,\dots)$, the
corresponding torus fixed point is the monomial ideal
\begin{equation}
  I_\lambda 
  = (\xi^{\lambda_1}, z\xi^{\lambda_2},\dots, z^{i-1}\xi^{\lambda_i},\dots).
\end{equation}
If we put the monomial $z^{i-1} \xi^{j-1}$ in the box at the
intersection of the $i^{\mathrm{th}}$ column and the $j^{\mathrm{th}}$
row, $I_\lambda$ is generated by monomials which sit outside of
$D$. See Figure~\ref{fig:Young}.

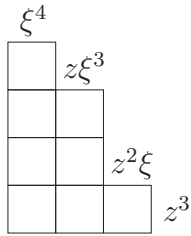
\begin{figure}[htbp]
\begin{center}
\leavevmode
\setlength{\unitlength}{0.01250000in}
\begin{picture}(73,114)(0,-10)
\path(22,62)(22,82)(2,82)
        (2,62)(22,62)
\path(22,42)(22,62)(2,62)
        (2,42)(22,42)
\path(42,42)(42,62)(22,62)
        (22,42)(42,42)
\path(22,22)(22,42)(2,42)
        (2,22)(22,22)
\path(42,22)(42,42)(22,42)
        (22,22)(42,22)
\path(22,2)(22,22)(2,22)
        (2,2)(22,2)
\path(42,2)(42,22)(22,22)
        (22,2)(42,2)
\path(62,2)(62,22)(42,22)
        (42,2)(62,2)
\put(7,87){\makebox(0,0)[lb]{\smash{$\xi^4$}}}
\put(25,67){\makebox(0,0)[lb]{\smash{$z\xi^3$}}}
\put(45,27){\makebox(0,0)[lb]{\smash{$z^2\xi$}}}
\put(67,7){\makebox(0,0)[lb]{\smash{$z^3$}}}
\end{picture}
\caption{Young diagram and an ideal}
\label{fig:Young}
\end{center}
\end{figure}

\subsection{Symmetric products of the $x$-axis}\label{sec:symmetric-products-x}

The constant $c_{i,n}$ (in fact, it is easy to see that it is
independent of $n$ by a similar argument as above) was determined by
studying symmetric products of an embedded curve $C$ in $X$ in the
nonequivariant case \cite[Ch.9]{Lecture}. We shall briefly review the
study here, emphasizing that the argument goes through even when
$C$ and $X$ are noncompact. We shall use the construction later.

Let $C$ denote the $z$-axis, i.e., $C = \{ \xi =0\}$. Let $L^*C$
denote the subvariety consisting of $I\in\bigsqcup_n \Hilb{n}$ such
that $\CC[z,\xi]/I$ is contained in $C$. In other words, it is the
union $\bigsqcup_n \pi^{-1}(S^nC)$ of inverse images of the symmetric
products $S^n C$ under the Hilbert-Chow morphism.

By \cite[Ch.7]{Lecture} it has irreducible components $L^\lambda C$
parametrized by partitions $\lambda$. Moreover, it is a lagrangian
subvariety in $\Hilb{n}$ with $n=|\lambda|$. We recall two
descriptions of irreducible component $L^\lambda C$.

The first one is by the $\CC^*$-action given by $t\cdot (z,\xi)\mapsto
(z,t\xi)$. Then $I\in \CC^*$ is fixed by $\CC^*$ if and only if it is
of a form
\begin{equation}
  I = I_{D_1,z_1}\cap \cdots\cap I_{D_k,z_k},
\end{equation}
where $z_k\in C$, $D_i$ is a Young diagram and $I_{D_i,z_i}$ is the
ideal generated by monomials in $(z-z_i)$ and $\xi$ which sit outside
of $D_i$ as in \subsecref{sec:torus}. The diagram $D_i$ and the point
$z_i$ is uniquely determined by $I$. Let $D$ be the union of all
$D_i$, reordering columns so that it is a Young diagram. Let $\lambda$
be the corresponding partition. Then $S^\lambda C$ consisting of fixed
points $I$ with $D = D_1\sqcup D_2\sqcup\cdots$ is a connected
component of $(\Hilb{n})^{\CC^*}$, and any connected component is of
this form.

We have
\begin{equation}
  L^* C = \bigsqcup_n 
  \left\{ I\in\Hilb{n} \, \middle|\, \text{$\lim_{t\to\infty} t\cdot I$
      exists} \right\},
\end{equation}
which decomposes according to the limit as $L^*C = \bigsqcup
W_\lambda^-$ with
\begin{equation}\label{eq:11}
  W_\lambda^- \defeq \left\{ I\in L^* C \,\middle|\, 
    \lim_{t\to\infty} t\cdot I \in S^\lambda C \right\}.
\end{equation}
Each $W_\lambda^-$ is a locally closed lagrangian subvariety in
$\Hilb{n}$ and we define 
\begin{equation}
  L^\lambda C \defeq \text{Closure of }W_\lambda^-.  
\end{equation}

For example, $\lambda = (1^n)$, $L^\lambda C$ is isomorphic to the
symmetric product $S^nC$ of $C$, embedded into $\Hilb{n}$ via the
natural morphism $S^nC \cong \Hilb[C]{n}\to \Hilb{n}$. This is clearly
a lagrangian subvariety in $\Hilb{n}$ since the symplectic form is an
extension of one on the open subset $\pi^{-1}(S^n_{(1^n)}X)$. The
opposite extreme is $L^{(n)}C$, which consists of ideals $I$ such that
$\CC[z,\xi]/I$ is supported at a single point in $C$.
The lagrangian property is less clear, and follows from the Morse
theoretic argument in \cite[Ch.7]{Lecture}.

The second description is by the Hilbert-Chow morphism $\pi\colon
\Hilb{n}\to S^n X$. Recall we have $L^* C = \pi^{-1}(S^n C)$. Note
that $S^n C$ has the natural stratification $S^n C = \bigsqcup
S^n_\lambda C$, where
\begin{equation}
  S^n_\lambda C = \left\{ \sum_i \lambda_i [x_i]\in S^n C\, \middle|\,
    \text{$x_i\neq x_j$ for $i\neq j$}\right\}.
\end{equation}
We decompose $L^*C$ accordingly
\begin{equation}
  L^*C = \bigsqcup \pi^{-1}(S^n_\lambda C).
\end{equation}
Each $\pi^{-1}(S^n_\lambda C)$ is a locally closed $n$-dimensional
irreducible subvariety in $L^* C(\cap\Hilb{n})$, where the
irreducibility follows from that of punctual Hilbert schemes
$\pi^{-1}(m[p])$ (\cite[Th.~5.12]{Lecture}). Hence its closure is an
irreducible component.

\begin{Proposition}
  $L^\lambda C$ is the closure of $\pi^{-1}(S^n_\lambda C)$.
\end{Proposition}

\begin{proof}
  It is enough to note that the above component $S^\lambda C$ of
  $(\Hilb{n})^{\CC^*}$ has an open locus consisting of $I =
  I_{D_1,z_1}\cap\cdots \cap I_{D_k,z_k}$ such that all $D_i$'s have
  only single column.
\end{proof}

There is the third description, which was not given in \cite{Lecture}:
Let us consider the {\it hyperbolic\/} $\CC^*$-action $t\ast(z,\xi) =
(t^{-1}z, t\xi)$. Then we claim that
\(
   \idl\in \bigsqcup_n \Hilb{n}
\)
has limit when $t\to \infty$ if and only if $\idl\in L^*C$. If we
replace $\Hilb{n}$ by the symmetric product $S^nX$, the corresponding
assertion is obvious. A point $\sigma\in S^nX$ has limit when $t\to
\infty$ if and only if $\sigma\in S^n C$. Since $\pi$ is proper, the
assertion follows. Then $L^*C$ decomposes according to the
decomposition of $(\Hilb{n})^{\CC^*}$ as in \eqref{eq:11}, where the
$\CC^*$-action is replaced by the hyperbolic one.
Let us first note that $(\Hilb{n})^{\CC^*} = (\Hilb{n})^T$. This
follows, for example, from the $T$-character formula of the tangent
space at $I_\lambda\in(\Hilb{n})^T$ reviewed in \propref{prop:char}
below. Put $t_1 = t_2^{-1} = t$. We see that $T_{I_\lambda}\Hilb{n}$
has a trivial weight zero space. It means that $(\Hilb{n})^{\CC^*}$
cannot be larger than $(\Hilb{n})^T$.

\begin{Proposition}
\begin{equation*}
  W^-_\lambda = \left\{ I\in L^*C \, \middle|\,
      \lim_{t\to\infty} t\ast I = I_\lambda\right\}.
\end{equation*}
\end{Proposition}

\begin{proof}
  Since both left and right hand sides are $T$-invariant locally
  closed submanifolds, it is enough to prove the equality in a
  neighborhood of the fixed point $I_\lambda$. 

  The tangent space of $W^-_\lambda$ at $I_\lambda$ is equal to the
  sum of nonpositive weight spaces in $T_{I_\lambda}\Hilb{n}$ with
  respect to the $\CC^*$-action given by $t\cdot (z,\xi) =
  (z,t\xi)$. Similarly, the tangent space of the right hand side is
  the sum of negative weight spaces with respect to the hyperbolic
  $\CC^*$-action $t\ast (z,\xi) = (t^{-1}z, t\xi)$. Looking at the
  formula in \propref{prop:char} below, one finds that both spaces are
  the same space. In fact, it corresponds to $\sum t_1^{l(s)+1}
  t_2^{-a(s)}$. (One can also check that it is a lagrangian subspace
  in $T_{I_\lambda}\Hilb{n}$, as the symplectic form is of weight $t_1
  t_2$.)

  There is no nontrivial $T$-homomorphism from any symmetric power of
  $T_{I_\lambda} W^-_\lambda$ to the sum of complementary weight
  spaces, and hence a $T$-invariant submanifold with the tangent space
  $T_{I_\lambda} W^-_\lambda$ is unique, as in the proof of the
  existence of Bia{\l}ynicki-Birula decomposition (see
  \cite[Th.~2.2]{BB}).
  \begin{NB}
    We take a $T$-equivariant chart around $I_\lambda$.
  \end{NB}%
  Therefore two submanifolds are equal.
\end{proof}

Let us consider the top degree Borel-Moore homology group of $L^*C$:
\begin{equation}
  H_{\operatorname{top}}^{lf}(L^*C)
  = \bigoplus_n H_{2n}^{lf}(L^*C\cap\Hilb{n}).
\end{equation}
This is a vector space with a base $\{ [L^\lambda C]\}$.
The creation operator $\PP{\beta}{-i}$ acts on
$H_{\operatorname{top}}^{lf}(L^*C)$ for $\beta = [C]\in
H_{2}^{lf}(C)$. On the other hand, for the annihilation operator
$\PP{\alpha}{i}$, we take $\alpha = [\text{$y$-axis}]$. Then the
intersection pairing of $\alpha$ and $\beta$ is well-defined (in fact,
it is $1$), as $\{\text{$x$-axis}\}\cap\{\text{$y$-axis}\}$ is a
single point in $X$, and $\PP{\alpha}{i}$ is a well-defined operator on
$H_{\operatorname{top}}^{lf}(L^*C)$. The top degree is preserved by
the convolution product, as $P[i]$ in \eqref{eq:6} is middle
dimensional in $\Hilb{n}\times\Hilb{n-i}\times X$, and $\alpha$,
$\beta$ are so in $X$.

We have a linear map from the ring $\Lambda$ of symmetric functions to
$H_{\operatorname{top}}^{lf}(L^*C)$ by
\begin{equation}
  p_\lambda = p_{\lambda_1} p_{\lambda_2}\cdots
  \mapsto \PP{\beta}{-\lambda_1} \PP{\beta}{-\lambda_2}\cdots 1,
\end{equation}
for $\lambda = (\lambda_1,\lambda_2,\dots)$. From the Heisenberg
algebra relation \eqref{eq:5}, it is injective. (For this we only need
that the constant $c_{i,n}$ is nonzero, which is a consequence of the
Poincar\'e duality.) Since $\dim H_{\operatorname{top}}^{lf}(L^*C\cap
\Hilb{n})$ is equal to the number of partitions of $n$, it is an
isomorphism.

In order to determine the coefficient $i$ in \eqref{eq:5}, it is
enough to compute it in the current situation. A key result is
\begin{Proposition}\label{prop:monomial}
  Under the isomorphism $\Lambda\cong
  H_{\operatorname{top}}^{lf}(L^*C)$, the class $[L^\lambda C]$
  corresponds to the monomial symmetric function $m_\lambda$.
\end{Proposition}

This result together with the formula
\begin{equation}
  \PP{\alpha}{i} [L^{(1^n)}C] = [L^{(1^{n-i})}C]
\end{equation}
(see \cite[Lemma~9.21]{Lecture}) determines the coefficient.

\begin{NB}
  We may write yet another characterization as unstable manifolds with
  respect to the hyperbolic action $t\cdot(z,\xi) = (tz,t^{-1}\xi)$
\end{NB}

\section{Jack symmetric functions and torus fixed
  points}\label{sec:Jack}

The goal of this section is to give a geometric realization of Jack
symmetric functions in the $T$-equivariant homology groups of Hilbert
schemes of points in $X=\CC^2$. 

The result presented here was stated and proved by Li-Qin-Wang
\cite{LQW-jack}. Both the framework and the proof were based on an
earlier work by Vasserot \cite{VasserotC2}, who considered Schur
functions and the $\CC^*$-equivariant homology groups, instead of
$T$-equivariant ones.
Vasserot's work, in turn, was motivated by author's unpublished paper
\cite{Jack}, where Jack symmetric functions were considered in a
different setting, i.e., the case when $X$ is the total space of a
line bundle over a Riemann surface.

Here we present the proof in \cite{LQW-jack}, as well as materials in
\cite{Jack} used there.

The result here could be viewed as a homological version of Haiman's
result \cite{Haiman}, which relates Macdonald polynomials to
$T$-equivariant $K$-theory of Hilbert schemes of points in
$X$. However there is a big difference: we relate $H^T_*(\Hilb{n})$
with the ring of symmetric functions via the Heisenberg representation
in the previous section. On the other hand, Haiman relates them via
the {\it Procesi bundle\/}, a rank $n!$ vector bundle over
$\Hilb{n}$. The definition of the Procesi bundle is rather involved,
and will not be presented here. Let us also remark that the Procesi
bundle is essential to prove the positivity conjecture for Macdonald
polynomials, and we do not have a counter part in our theory.

\subsection{Jack symmetric functions}

Let us briefly recall the definition of Jack symmetric functions in
\cite[VI.10]{Mac}. Our notation follows \cite{Mac} except that the
parameter $\alpha$ is replaced by $\bk$ here.

Let $\bk$ be an indeterminate. We define an inner product on the ring
$\Lambda_{\QQ(\bk)}$ of symmetric functions with coefficients in
$\QQ(\bk)$ by
\begin{equation}\label{eq:3}
  \la p_\lambda, p_\mu\ra = \delta_{\lambda\mu} \bk^{l(\lambda)} z_\lambda,
\end{equation}
where $l(\lambda)$ is the length of a partition $\lambda$ and
$z_\lambda = \prod k^{m_k} m_k!$ for $\lambda = (1^{m_1}
2^{m_2}\cdots)$.

Let $\ge$ denote the {\it dominance order\/} on partitions, i.e.,
$\lambda\ge\mu$ if and only if $|\lambda|=|\mu|$ and
\begin{equation}\label{eq:dom}
  \lambda_1+\cdots+\lambda_i \ge \mu_1 + \cdots + \mu_i 
  \quad\text{for all $i$}.
\end{equation}

If $\lambda'$, $\mu'$ denote conjugate partitions of $\lambda$, $\mu$
respectively, we have
\begin{equation*}
  \lambda\ge\mu \Longleftrightarrow \mu'\ge\lambda'.
\end{equation*}
See \cite[I.(1.11)]{Mac} for the proof.

Let $m_\lambda$ be the monomial symmetric function. A Jack symmetric
function $P^{(\bk)}_\lambda$ is characterized by the following two
properties
\begin{align}
  & P^{(\bk)}_\lambda = m_\lambda
  + \sum_{\mu < \lambda} u^{(\bk)}_{\lambda\mu} m_\mu\quad
  \text{for $u^{(\bk)}_{\lambda\mu}\in\QQ(\bk)$},\label{eq:2}
\\
  & \la P^{(\bk)}_\lambda, P^{(\bk)}_\mu\ra = 0 \quad
  \text{if $\lambda\neq\mu$}.\label{eq:2.1}
\end{align}

Let $\succeq$ be a total order refining the dominance order. Then $\{
P^{(\bk)}_\lambda\}$ is obtained from the base $\{ m_\lambda\}$ by
Gram–-Schmidt orthogonalization with respect to $\succeq$. Hence the
uniqueness of $P^{(\bk)}_\lambda$ follows. The existence is obvious if
we replace $\ge$ by $\succeq$ in \eqref{eq:2}, and this difference of
orders is a nontrivial part in the theory of Jack symmetric
functions. We briefly recall the proof in \cite{Mac} in the next
subsection.

When $\bk = 1$, the inner product \eqref{eq:3} is the standard
one. The above two properties are satisfied by Schur functions
$s_\lambda$.

\subsection{Hamiltonian}

For a positive integer $N$, let us consider the ring
$\Lambda_{N,\QQ(\bk)}$ of symmetric functions in $N$-variables $x_1$,
\dots, $x_N$ with $\QQ(\bk)$-coefficients. (Note that $N$, $\bk$
correspond to $n$, $\alpha$ in \cite{Mac}.)

Let $X$ be an indeterminate and let
\begin{equation*}
  D_N \defeq a_\delta(x)^{-1}\sum_{w\in S_N} \varepsilon(w) x^{w\delta}
  \prod_{i=1}^N 
  \left(X + (w\delta)_i + \bk x_i\frac{\partial}{\partial x_i}\right),
\end{equation*}
where $\delta = (N-1,N-2,\dots,0)$, $a_\delta(x) =
\prod_{i<j}(x_i-x_j)$ is the Vandermonde determinant, $\ve(w)=\pm 1$
is the sign of $w\in S_N$, and $(w\delta)_i$ is the $i^{\mathrm{th}}$
component of $w\delta$.

For $r=0,1,\dots, N$ let $D^r_N$ denote the coefficient of $X^{N-r}$
in $D_N$:
\begin{equation*}
  D_N = \sum_{r=0}^N X^{N-r} D^r_N.
\end{equation*}

By \cite[VI, \S3.~Ex.3(a)]{Mac} we have
\begin{equation*}
  D_N m_\lambda(x) = \sum_{w\in S_N}\prod_{i=1}^N (X + N-i + \bk (w\delta)_i)
  s_{w\lambda}(x),
\end{equation*}
where
\begin{equation*}
  s_{w\lambda}(x) = \frac{a_{\delta+w\lambda}(x)}{a_{\delta}(x)}
  = \frac{1}{a_{\delta}(x)}
  \sum_{w_1\in S_N} \ve(w_1) x^{w_1(w\lambda+\delta)}.
\end{equation*}
(There is a typo in \cite[VI, \S3,~Ex.3(a)]{Mac}. In the formula (a),
$\lambda_i$ should read $\beta_i$ or $(w\lambda)_i$ in the current
notation.)
\begin{NB}
  We take the differential of $D_n(X;q,t)$ at $t=1$ after setting
  $q=t^\alpha$. Then we obtain the above
  $\left.X^{-N}D_n\right|_{X\mapsto X^{-1}}$. Therefore the above
  formula follows from \cite[VI (3.6)]{Mac}.
\end{NB}%
Since $s_{w\lambda}(x)$ is either zero or is equal to $\pm s_\mu$ for
some partition $\mu < \lambda$ (unless $w=1$), and since the transition
matrix between $\{ s_\lambda \}$ and $\{ m_\lambda\}$ is upper
triangular with $1$ on diagonal, we get
\begin{equation*}
  D_N m_\lambda(x) = \sum_{\mu\le\lambda} c_{\lambda\mu}(X;\bk) m_\mu(x)
\end{equation*}
with $c_{\lambda\mu}\in \ZZ[X,\bk]$. See \cite[VI,
\S4,~Ex.4]{Mac}. Here $m_\lambda(x)$ and $s_\lambda(x)$ are zero if
$l(\lambda) > N$, and $\{ m_\lambda(x)\}_{l(\lambda)\le N}$, $\{
s_\lambda(x)\}_{l(\lambda)\le N}$ are bases of $\Lambda_{N}$.  The
diagonal entries are given by
\begin{equation*}
  c_{\lambda\lambda}(X;\bk) = \prod_{i=1}^N (X + N-i + \bk\lambda_i).
\end{equation*}

\begin{NB}
  Let us consider the coefficient $c^2_{\lambda\lambda}$ of $X^{N-2}$:
  \begin{equation*}
    \begin{split}
      &c^2_{\lambda\lambda}
    = \sum_{i<j} (N-i +\bk\lambda_i)(N-j +\bk\lambda_j)
\\
   =\; & \frac12\sum_{i\neq j}(N-i)(N-j)+2\bk\lambda_i(N-j)
   + \bk\lambda_i\lambda_j.
    \end{split}
  \end{equation*}
Then
\begin{equation*}
  \begin{split}
  & \sum_{i\neq j} \lambda_i\lambda_j = |\lambda|^2 - \sum \lambda_i^2
  = |\lambda|^2 - |\lambda| -  2 n(\lambda^t) 
  = |\lambda|^2 - |\lambda| -  2 \sum_{s\in \lambda} a'(s),
\\
  & \sum_{i\neq j} j \lambda_i = \frac{N(N-1)}2|\lambda| - \sum_j j\lambda_j
  = \left(\frac{N(N-1)}2-1\right)|\lambda| - n(\lambda)
  = \left(\frac{N(N-1)}2-1\right)|\lambda| - \sum_{s\in\lambda} l'(s).
  \end{split}
\end{equation*}
\end{NB}

When we set $\bk$ a positive real number, $N - i + \bk\lambda_i$ is
strictly decreasing in $i$. Therefore $c_{\lambda\lambda}\neq
c_{\mu\mu}$ for $\lambda\neq\mu$. Therefore
\begin{NB}
  See the argument in \cite[VI (4.7)]{Mac}.
\end{NB}%
there exists $P^{(\bk)}_\lambda$ of the form \eqref{eq:2} such that
\begin{equation*}
  D_N P^{(\bk)}_\lambda = c_{\lambda\lambda}(X;\bk) P^{(\bk)}_\lambda.
\end{equation*}

Moreover $D_N$ is self-adjoint with respect to the finitely many
variable version of the inner product \eqref{eq:3} \cite[VI,
\S3,~Ex.3(b)]{Mac}. Therefore we deduce \eqref{eq:2.1}.
\begin{NB}
\begin{equation*}
  \la P^{(\bk)}_\lambda, P^{(\bk)}_\mu\ra = 0\quad\text{if $\lambda\neq\mu$}.
\end{equation*}
\end{NB}%
It means that $\{ P^{(\bk)}_\lambda\}$ is obtained from the base $\{
m_\lambda\}$ by Gram–-Schmidt orthogonalization with respect to
any total order $\succeq$ compatible with $\ge$.

The ring $\Lambda_{\QQ(\bk)}$ is defined as the direct sum of
projective limit
\begin{equation*}
  \Lambda_{\QQ(\bk)} = \bigoplus \varprojlim_N \Lambda^n_{N,\QQ(\bk)},
\end{equation*}
where $\Lambda^n_{N,\QQ(\bk)}$ is the degree $n$ part of
$\Lambda_{N,\QQ(\bk)}$, and the inverse system is given by the
homomorphism $\Lambda_{M,\QQ(\bk)}\to \Lambda_{N,\QQ(\bk)}$ sending
$x_{N+1},\dots, x_{M}$ to $0$.
Under the homomorphism $m_\lambda(x_1,\dots,x_M)$ is sent to
$m_\lambda(x_1,\dots,x_N)$ if $l(\lambda)\le N$, and to $0$ if
$l(\lambda) > N$. Therefore Gram--Schmidt orthogonalization is
compatible with $N$, and we get
$P^{(\bk)}_\lambda\in\Lambda_{\QQ(\bk)}$, as the limit of
$P^{(\bk)}_\lambda$ above when $N\to\infty$. This finishes the proof
of the existence of Jack symmetric functions.

For a later purpose, let us give a formula for $D^2_N$ (\cite[VI, \S3,
Ex.3(d)]{Mac}). Suppose that $f$ is a homogeneous polynomial of degree
$r$. Then
\begin{equation*}
  D_N^2 f = (-\bk^{2} U_N - \bk V_N + c_N) f,
\end{equation*}
where
\begin{equation*}
  U_N = \frac12 \sum_{i=1}^N x_i^2 \frac{\partial^2}{\partial x_i^2}, \quad
  V_N = \sum_{i\neq j} \frac{x_i^2}{x_i - x_j}\frac{\partial}{\partial x_i},
\end{equation*}
and
\begin{equation*}
  c_N = \frac12 \bk^{2} r(r-1) + \frac12 \bk rN(N-1)
  +\frac1{24} N(N-1)(N-2)(3N-1).
\end{equation*}

We further introduce an operator
\begin{equation*}
  \square^{\bk}_N f = \left(\bk U_N + V_N - (N-1)r\right)f.
\end{equation*}
This operator has the limit 
\begin{equation}\label{eq:limit}
  \square^{\bk} = \varprojlim_{N}\square^{\bk}_N.
\end{equation}
See \cite[VI, \S4, Ex.3(a)]{Mac} or \eqref{eq:15} below.

We have
\begin{equation}\label{eq:17}
  \square^{\bk} P^{(\bk)}_\lambda = e_\lambda(\bk) P^{(\bk)}_\lambda;
  \qquad e_\lambda(\bk)\defeq n(\lambda')\bk - n(\lambda),
\end{equation}
where
\begin{equation*}
  n(\lambda) = \sum (i-1) \lambda_i = \sum \frac{\lambda_i'(\lambda'_i-1)}2.
\end{equation*}

Computing how $\square^{\bk}$ acts on the base $\{ p_\lambda\}$ of
$\Lambda_{\QQ(\bk)}$, we obtain the following formula:
\begin{equation}\label{eq:15}
  \begin{split}
  \square^{\bk} &=
  \frac{\bk}2 \sum_{m,n>0} m n p_{m+n} \frac{\partial}{\partial p_m}
  \frac{\partial}{\partial p_n}
  + \frac{\bk-1}2 \sum_{m>0} m(m-1)p_m \frac{\partial}{\partial p_m}
\\
  &\qquad +
  \frac12\sum_{m,n>0} (m+n) p_m p_n \frac{\partial}{\partial p_{m+n}}.
  \end{split}
\end{equation}
Here we regard $\Lambda_{\QQ(\bk)}$ as a polynomial ring
$\QQ(\bk)[p_1,p_2,\dots]$. This formula will be crucial in
\secref{sec:Vir}. It is not present in \cite{Mac}. The author learned
it from \cite{AMOS} when he wrote \cite{Jack}, but it was certainly
known much before.

\begin{NB}
\begin{equation*}
  \begin{split}
  U_N &= \frac12\sum_{m,n>0} m n p_{m+n} \frac{\partial}{\partial p_m}
  \frac{\partial}{\partial p_n}
  + \frac12 \sum_{m>0} m(m-1)p_m \frac{\partial}{\partial p_m},
\\
  V_N &= \frac12\sum_{m,n>0} (m+n) p_m p_n \frac{\partial}{\partial p_{m+n}}
  + \frac12 \sum_{m>0} (2N-1-m)m p_m\frac{\partial}{\partial p_{m}}.
  \end{split}
\end{equation*}
\end{NB}

\begin{Remark}
    The operator $D^2_N$ is essentially equal to the
    Calogero-Sutherland hamiltonian, which has been studied
    intensively in the context of quantum integrable systems.
    (See \cite[\S5.5]{Etingof-CM} and the reference therein for example.)
\begin{NB}
Here the Calogero-Sutherland hamiltonian:
\begin{equation*}
    H_N \defeq \sum_{i=1}^N D_i^2 + \bk\sum_{i<j}\frac{x_i+x_j}{x_i-x_j}(D_i - D_j),
    \qquad D_i\defeq x_i\frac{\partial}{\partial x_i}.
\end{equation*}
\end{NB}%
It is a trigonometric analog of a quantization of the Calogero-Moser
integrable system, appeared in Wilson's work \cite{Wilson}, mentioned
at \cite[a paragraph preceding Theorem~3.46]{Lecture}.
At first sight, two appearances of this integrable system have no
link: one is a classical system and appears in the deformation of
$\Hilb{n}$, while the other is quantum and appears in the cohomology
of $\Hilb{n}$. However, they are connected in a deep way:
Bezrukavnikov-Finkelberg-Ginzburg \cite{BFG-Hilb} considered the
quantized integrable system in {\it positive characteristic}, and
connected it with the derived category of $\Hilb{n}$. See also
\cite{2012arXiv1208.3696B} for a nice application of this result.
\end{Remark}

\subsection{Inner product}

After giving the review of the definition and basic properties of Jack
symmetric functions, we start study of equivariant homology groups of
Hilbert schemes.
In this subsection we identify \eqref{eq:3} with the inner product on
$\bigoplus H^*_T(\Hilb{n})_\loc$ induced from \eqref{eq:14}.

Let $\alpha,\beta\in H^*_T(X)_\loc$. The commutation relation
\eqref{eq:5} implies
\begin{equation*}
  \la \PP{\alpha}{-n}, \PP{\beta}{-n}\ra = n\la \alpha,\beta\ra.
\end{equation*}

More generally, we consider an analog of $p_\lambda$:
\begin{equation*}
  \PP{\alpha}{\lambda} = \PP{\alpha}{-\lambda_1}\PP{\alpha}{-\lambda_2}\cdots
\end{equation*}
for $\lambda = (\lambda_1,\lambda_2,\dots)$. Then we have
\begin{equation}\label{eq:1}
  \la \PP{\alpha}{\lambda}, \PP{\beta}{\mu}\ra
  = \delta_{\lambda\mu}\la\alpha,\beta\ra^{l(\lambda)} z_\lambda,
\end{equation}
where $z_\lambda = \prod k^{m_k} m_k!$ for $\lambda = (1^{m_1}
2^{m_2}\cdots)$.
\begin{NB}
\(  \left(\frac{\partial}{\partial x}\right)^n x^n = n!.\)
\end{NB}

In particular, we take $\alpha$ and $\beta$ to be the Poincar\'e dual
of the $x$-axis, i.e., $\alpha = \beta = \ve_2$. Then
\begin{equation*}
  \la\alpha,\beta\ra
  \begin{NB}
    = - \frac{\ve_2^2}{\ve_1\ve_2}
  \end{NB}%
  = - \frac{\ve_2}{\ve_1}.
\end{equation*}

Substituting this into \eqref{eq:1} and comparing the result with
\eqref{eq:3}, we find
\begin{Proposition}\label{prop:ident}
  Our inner product is equal to \eqref{eq:3} used to define Jack
  symmetric functions under the identification $\PP{\ve_2}{\lambda}
  \leftrightarrow p_\lambda$, where the parameter $\bk$ is
  $-\ve_2/\ve_1$.
\end{Proposition}

\begin{NB}
  In this situation \eqref{eq:5} reads
  \begin{equation*}
  \left[ \PP{\alpha}{i}, \PP{\beta}{j}\right]
  = i\delta_{i+j,0} \bk\operatorname{id}.
  \end{equation*}
\end{NB}

When $\ve_1+\ve_2 = 0$, we have $\bk = -\ve_2/\ve_1 = 1$. It means
that our inner product is the standard inner product on symmetric
polynomials.

\subsection{Dominance order}

Recall that we identify the monomial symmetric function $m_\lambda$
with the class $[L^\lambda C]$ in \propref{prop:monomial}. In order to
understand the characterization of Jack symmetric functions in
(\ref{eq:2}, \ref{eq:2.1}), our next task is to explain a geometric
meaning of the dominance order \eqref{eq:dom}. It is given by
modifying the stratification introduced in \cite{Bri,Iar}.

Let $C = \{ \xi=0\}$ as before.
For $i\ge 0$, let $(\xi^i)
\begin{NB}
  = \shfO_X(-iC)
\end{NB}%
$ be the ideal of functions vanishing to order $\ge i$ along $C$.
Let $\idl\in\Hilb{n}$ be an ideal of colength $n$ such that the
support of $\CC[z,\xi]/\idl$ is contained in $C$.  We consider the
sequence $(\lambda'_1,\lambda'_2,\dots)$ of nonnegative integers given
by
\begin{equation}\label{eq:9}
  \lambda'_i(\idl)\defeq
   \begin{NB}
  \operatorname{length}\left(
   \frac{\shfO_X(-(i-1)C)}{\idl\cap\shfO_X(-(i-1)C)+\shfO_X(-iC)}
   \right) =
    \end{NB}%
  \dim\left(
    \frac{(\xi^{i-1})}{\idl\cap (\xi^{i-1})+ (\xi^i)}
  \right).
\end{equation}
The reason why we put the prime become clear later.  The sequence in
\cite{Bri,Iar} was defined by replacing $(\xi^i)
\begin{NB}
= \shfO_X(-iC)  
\end{NB}%
$ by $\frak m_x^i$ where $\frak m_x$ is the maximal ideal
corresponding to a point $x$. It is clear that $\idl\supset (\xi^n)
\begin{NB}
  = \shfO_X(-nC)
\end{NB}%
$ (see e.g., \cite[Lemma~1.1]{Iar}), hence $\lambda'_i(\idl) = 0$ for
$i\ge n+1$.  From the exact sequence
\begin{NB}
\begin{equation*}
 0 \to \frac{\shfO_X(-iC)}{\idl\cap\shfO_X(-iC)} \to
      \frac{\shfO_X(-(i-1)C)}{\idl\cap\shfO_X(-(i-1)C)} \to
 \frac{\shfO_X(-(i-1)C)}{\idl\cap\shfO_X(-(i-1)C)+\shfO_X(-iC)}
 \to 0,
\end{equation*}
\end{NB}%
\begin{equation*}
 0 \to \frac{(\xi^i)}{\idl\cap (\xi^i)} \to
      \frac{(\xi^{i-1})}{\idl\cap (\xi^{i-1})} \to
 \frac{(\xi^{i-1})}{\idl\cap (\xi^{i-1})+(\xi^i)}
 \to 0,
\end{equation*}
we have
\begin{equation*}
  \sum_{i=1}^n \lambda'_i(\idl) = n.
\end{equation*}

\begin{NB}
The following deduction is unnecessary for $X=\CC^2$.

Let us decompose the ideal $\idl$ by its support, i.e.,
$\idl = \idl_1\cap\dots\cap\idl_N$ such that
$\{ \Supp(\CC[z,\xi]/\idl_k) \}_k$ are $N$ distinct points.
By definition, we have
\begin{equation}
  \lambda_i'(\idl) = \sum_{k=1}^N \lambda_i'(\idl_k).
\label{eq:sum}\end{equation}
\end{NB}

\begin{NB}
Suppose that $\idl$ satisfies $\Supp(\CC[z,\xi]/\idl) = \{(z,0)\}$ for
some $(z,0)\in C$.
\end{NB}%
If $\xi^i f_1(z),\dots, \xi^i f_d(z)$ form a basis of
\begin{equation*}
  \frac{(\xi^i)}{\idl\cap (\xi^i)+ (\xi^{i+1})}
  \begin{NB}
    =
    \frac{\shfO_X(-iC)}{\idl\cap\shfO_X(-iC)+\shfO_X(-(i+1)C)}
  \end{NB}%
  ,  
\end{equation*}
Then $\xi^{i-1} f_1(z),\dots, \xi^{i-1} f_d(z)$ are linearly
independent in 
\begin{equation*}
    \frac{(\xi^{i-1})}{\idl\cap (\xi^{i-1})+(\xi^i)}
  \begin{NB}
    \frac{\shfO_X(-(i-1)C)}{\idl\cap\shfO_X(-(i-1)C)+\shfO_X(-iC)}
  \end{NB}%
  .
\end{equation*}
Hence we have $\lambda'_i(\idl) \ge \lambda'_{i+1}(\idl)$. Thus
$(\lambda'_1(\idl), \lambda'_2(\idl), \dots)$ is a partition of $n$.
\begin{NB}
By \eqref{eq:sum}, the same is true for general $\idl$ which do not
necessarily satisfy $\Supp(\shfO_X/\idl)=\{x\}$.
\end{NB}%
Let us denote the partition by $\lambda'(\idl)$.

For a partition $\lambda' = (\lambda'_1,\lambda'_2,\dots)$ of $n$, let
$V^{\lambda'}$ be the set of ideals $\idl\in\Hilb{n}$ such that
$\CC[z,\xi]/\idl$ is supported on $C$ and $\lambda'(\idl) = \lambda'$.
Since
\begin{equation*}
  \dim \frac{(\xi^i)}{I\cap (\xi^i)}
  \begin{NB}
  = \operatorname{length}(\shfO_X(-iC)/\idl\cap\shfO_X(-iC))
  \end{NB}
  \leq
  \sum_{j=i+1}^n \lambda'_j = n - \sum_{j=1}^i \lambda'_j
\end{equation*}
is a closed condition on $\idl$,
\begin{NB}
  $\idl\cap (\xi^i)$ may drop the `dimension' when $\idl$ goes to a
  limit.
\end{NB}%
the union
\begin{equation*}
   \bigcup_{\mu': \mu'\ge \lambda'} V^{\mu'}
\end{equation*}
is a closed subset of $\{\idl\in \Hilb{n}\mid
\Supp(\CC[z,\xi]/\idl)\subset C\}$.
\begin{NB}
  For example, $\lambda'=(n)$ is the maximum, so $\mu'\ge \lambda'$
  implies $\mu' = (n)$. Therefore $V^{(n)}$ is closed. In fact,
  $V^{(n)}$ is the symmetric product of $C$. The opposite extreme is
  $\lambda' = (1^n)$. For example, if $n=2$, $I = (\alpha x+\beta y,
  x^2, \xi^2)$ is in $V^{(1^2)}$ if and only if $\alpha\neq 0$. Thus
  $V^{(1^2)}$ is not closed.
\end{NB}%
Thus we have
\begin{equation}
\label{eq:closure}
  \text{Closure of }V^{\lambda'} \subset\bigcup_{\mu'\ge\lambda'}
  V^{\mu'}.
\end{equation}

Suppose that $\lambda'$ is the conjugate partition of $\lambda$. We
get the following third description of $L^\lambda{C}$.
\begin{Proposition}
  $L^\lambda{C} = \text{Closure of }V^{\lambda'}$.
\label{prop:third}\end{Proposition}

\begin{proof}
  Let us write $\lambda = (\lambda_1,\dots,\lambda_N)$ with $N =
  l(\lambda)$.  Using the description of $L^\lambda C$ as the closure
  of $W_\lambda^-$, we first check that $W_\lambda^-\subset
  V^{\lambda'}$. We may take a generic point in $\idl\in W_\lambda^-$,
  so $\lim_{t\to\infty} t\cdot\idl = \idl_{(\lambda_1),z_1}\cap \dots
  \cap \idl_{(\lambda_N),z_N}$ such that $x_i$'s are distinct points
  in $C$ and Young diagrams $D_1$, \dots have only single column, $D_1
  = (\lambda_1)$, \dots.
  Since the support of $\idl$ cannot move as $t\to\infty$, we
  can decompose $\idl = \idl_1\cap\dots\cap\idl_N$ such that
  $\Supp(\CC[z,\xi]/\idl_k) = \{(z_k,0)\}$.

  Recall that $I_{(\lambda_k),z_k}$ is the ideal
  $(z-z_k,\xi^{\lambda_k})$. It is contained in $V^{({\lambda_k})'}$,
  as can be checked directly in the definition. Since $({\lambda_k})'
  = (1^{\lambda_k})$ is the unique minimum in the dominance order,
  $V^{({\lambda_k})'}$ is open in $L^*C$.
  As $\lim_{t\to\infty} t\cdot\idl_k = \idl_{(\lambda_k),x_k}$, we
  have $t\cdot \idl_k\in V^{({\lambda_k})'}$ for sufficiently large $t$.
  It is clear that $V^{\lambda'}$ is invariant under the
  $\CC^*$-action for any $\lambda$. In particular, $\idl_k\in
  V^{({\lambda_k})'}$. 
  Since $I=I_1\cap\dots\cap I_N$, we have $I\in V^{\lambda'}$. Thus
  $L^\lambda C\subset \text{Closure of }V^{\lambda'}$. As we have
  $L^*C = \bigsqcup V^{\lambda'}$,
  \begin{NB}
    So the closure of $V^{\lambda'}$ cannot be bigger than $L^*C$.
  \end{NB}%
  we conclude $L^\lambda C = \text{Closure of }V^{\lambda'}$.
\end{proof}

\subsection{Fixed points and Jack symmetric functions}

Recall that the torus fixed points in $\Hilb{n}$ are parametrized by
partitions $\lambda$ with $|\lambda|=n$, as we have explained at the
beginning of \subsecref{sec:heisenberg-algebra}. 

Let $\iota\colon (\Hilb{n})^T = \bigsqcup_\lambda \{ I_\lambda\}\to
L^*C \cap \Hilb{n}$ be the inclusion of the $T$-fixed point set. Here
note that all fixed points $I_\lambda$ are contained in $L^*C$. Let
$\zeta\colon L^*C \cap \Hilb{n}\to\Hilb{n}$ be the inclusion.
By the localization theorem, we have isomorphisms
\begin{equation}\label{eq:7}
  H^T_*((\Hilb{n})^T)_\loc \xrightarrow[\cong]{\iota_*}
  H^{T,lf}_*(L^*C \cap \Hilb{n})_\loc
  \xrightarrow[\cong]{\zeta_*}
  H^{T,lf}_*(\Hilb{n})_\loc.
\end{equation}

The leftmost space $H^T_*((\Hilb{n})^T)_\loc$ is the direct sum
$\bigoplus_\lambda \QQ(\ve_1,\ve_2)[I_\lambda]$, in particular, it has
a base $[I_\lambda]$. As $L^*C\cap\Hilb{n}$ is the union of
$W^-_\lambda$, which is a vector bundle over an affine space
$S^\lambda C$ by \subsecref{sec:symmetric-products-x}, a standard
argument (as in \cite[Ch.5]{Lecture}) show that $[L^\lambda C]$ is a
base of the middle space $H^{T,lf}_*(L^*C \cap \Hilb{n})_\loc$. In
particular, we have an isomorphism
\begin{equation}\label{eq:12}
  \Lambda\otimes\QQ(\ve_1,\ve_2) \cong H^{T,lf}_*(L^*C)_\loc,
\end{equation}
extending the isomorphism $\Lambda\cong H^{lf}_{\operatorname{top}}(L^*C)$.

Let $i_\lambda\colon \{ I_\lambda\}\to \Hilb{n}$ be the inclusion of
the fixed point $I_\lambda$ to the Hilbert scheme. Therefore
$\zeta\circ\iota = \bigsqcup i_\lambda$.

The inverse of the composition $\zeta_*\iota_*$ in \eqref{eq:7} is given by
\begin{equation*}
  \sum_\lambda \frac1{e(T_{I_\lambda}\Hilb{n})} i_\lambda^*(\bullet),
\end{equation*}
where $e(T_{I_\lambda}\Hilb{n})$ is the $T$-equivariant Euler class of
the tangent space of $\Hilb{n}$ at $I_\lambda$. See \thmref{thm:inverse}.

Let us consider the $\iota_*^{-1}([L^\lambda C])$. Note that the
partition given by \eqref{eq:9} for $I = I_\mu$ is $\mu'$. Hence
$I_\mu\in V^{\mu'}$. Therefore we have $I_\mu\in L^\lambda C$ only if
$\lambda\ge\mu$ by \eqref{eq:closure} and \propref{prop:third}. This
implies that
\begin{equation}\label{eq:10}
  \iota_*^{-1}([L^\lambda C]) \in 
  \bigoplus_{\mu\le\lambda} \QQ(\ve_1,\ve_2)[I_\mu].
\end{equation}

Let us next consider the coefficient of $[I_\lambda]$ in
\eqref{eq:10}. Recall that $L^\lambda C$ is defined as the closure of
$W_\lambda^-$ in \eqref{eq:11}. As $I_\lambda\in S^\lambda C$,
$L^\lambda C$ is a submanifold in a neighborhood of $I_\lambda$. The
tangent space $T_{I_\lambda}L^\lambda C$ is the direct sum of weight
subspaces of $T_{I_\lambda}\Hilb{n}$, whose weights are nonpositive
with respect to the $\CC^*$-action $t(z,\xi) = (z,t\xi)$. Let us
decompose the tangent space $T_{I_\lambda}\Hilb{n}$ into
$T^{>0}_{I_\lambda}\oplus T^{\le 0}_{I_\lambda}$, sum of positive and
nonpositive weight subspaces. Then the fiber of the normal bundle of
$L^\lambda C$ at $I_\lambda$ is identified with $T^{>0}_{I_\lambda}$.
Hence we have
\begin{equation}\label{eq:13}
  i_{\lambda*}^{-1}\zeta_*[L^\lambda C] 
  = \frac{e(T^{>0}_{I_\lambda})}{e(T_{I_\lambda}\Hilb{n})}[I_\lambda]
  = \frac1{e(T^{\le 0}_{I_\lambda})}[I_\lambda].
\end{equation}

Now we arrive at our main result in this section.
\begin{Theorem}\label{thm:Jack}
  Under the isomorphism \eqref{eq:7} together with \eqref{eq:12}, the
  class $\frac1{e(T^{\le 0}_{I_\lambda})}[I_\lambda]$ corresponds to
  the Jack symmetric function $P^{(\bk)}_\lambda$ with $\bk = -\ve_2/\ve_1$.
\end{Theorem}

\begin{proof}
  Let us check two properties \eqref{eq:2} and \eqref{eq:2.1}. The
  property \eqref{eq:2} follows from \eqref{eq:10} and \eqref{eq:13}.
  
  Next note that the composite $\zeta_*\iota_*$ of \eqref{eq:7}
  preserves the inner product by the definition of the inner product
  \eqref{eq:14}, where the inner product on
  $H^T_*((\Hilb{n})^T)_\loc\cong \bigoplus
  \QQ(\ve_1,\ve_2)[I_\lambda]$ is the direct sum of the standard inner
  product $\la [I_\lambda],[I_\lambda]\ra = 1$. Then it is clear that
  $\la \frac1{e(T^{\le 0}_{I_\lambda})}[I_\lambda], \frac1{e(T^{\le
      0}_{I_\mu})}[I_\mu]\ra = 0$ if $\lambda\neq\mu$.
\end{proof}

Let us make $e(T^{\le 0}_{I_\lambda})$ concrete.

\begin{Proposition}\label{prop:char}
  The character of the tangent space of $\Hilb{n}$ at the fixed point
  $I_\lambda$ is given by the formula
  \begin{equation*}
    \ch T_{I_\lambda}\Hilb{n} = 
    \sum_{s\in \lambda} \left(t_1^{l(s)+1} t_2^{-a(s)}
      + t_1^{-l(s)} t_2^{a(s)+1}\right).
  \end{equation*}
\end{Proposition}

See \cite[Prop.~5.8]{Lecture} for the proof.
Here $l(s)$, $a(s)$ are the {\it leg length\/} and the {\it arm
  length\/} of a square $s$ in the Young diagram corresponding to the
partition $\lambda$. Our convention is the same as in
\cite[(5.7)]{Lecture}, and also as in \cite[(6.14)]{Mac} except that
our Young diagram is rotated by $90^\circ$ in anti-clockwise. For a
later purpose, we also introduce the {\it leg colength\/} and the {\it
  arm colength\/} by $l'(s) = i-1$, $a'(s) = j-1$:
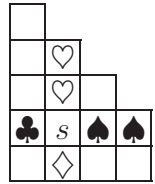
\begin{equation*}
\label{fig:hooklength}
\Yvcentermath1
\young(\hf,\hf\hs,\hf\hs\hf,\cs s\sps\sps,\hf\ds\hf\hf)\qquad\qquad
\begin{matrix}
 a(s) &= \text{number of $\hs$} \\
 l(s) &= \text{number of $\sps$} \\
 a'(s) &= \text{number of $\ds$} \\
 l'(s) &= \text{number of $\cs$}
\end{matrix}
\end{equation*}

\begin{Corollary}\label{cor:Euler}
  The equivariant Euler class of the nonpositive part $T^{\le
    0}_{I_\lambda}$ of the tangent space at the fixed point
  $I_\lambda$ is given by
  \begin{equation*}
    e(T^{\le 0}_{I_\lambda}) = \prod_{s\in\lambda}
    \left((l(s)+1)\ve_1 - a(s) \ve_2\right).
  \end{equation*}
  \begin{NB}
  The equivariant Euler class $e(T_{I_\lambda}\Hilb{n})$ of the
  tangent space at the fixed point $I_\lambda$ is given by
  \begin{equation*}
    e(T_{I_\lambda}\Hilb{n}) = \prod_{s\in\Lambda}
    \left((l(s)+1)\ve_1 - a(s) \ve_2\right)
    \left(-l(s)\ve_1 + (a(s)+1) \ve_2\right).
  \end{equation*}
\end{NB}
\end{Corollary}

Comparing this expression with \cite[VI, (10.21)]{Mac}, we
get
\begin{equation*}
    [I_\lambda]
    \begin{NB}
        =  e(T^{\le 0}_{I_\lambda}) P^{(\bk)}_\lambda
        = \ve_1^{|\lambda|} c_\lambda(\bk) P^{(\bk)}_\lambda
    \end{NB}%
    = \ve_1^{|\lambda|} J_\lambda^{(\bk)},
\end{equation*}
where $J_\lambda^{(\bk)}$ is the integral form of the Jack symmetric
function $P^{(\bk)}_\lambda$, defined in \cite[VI, (10.22)]{Mac}.

\begin{NB}
In particular, if we set $\ve_1 = -\ve_2 = \hbar$, we get
\begin{equation*}
    e(T_{I_\lambda}\Hilb{n}) = (-1)^{n} \hbar^{2n}
    \prod_{s\in\Lambda} h(s)^2.
\end{equation*}
The product of the hook lengths $h(s)$, denoted by $h(\lambda)$ in
\cite[I.1.\ Ex.10]{Mac}, appears in the representation theory of
symmetric groups $\mathfrak S_n$. For example, it is known that
$n!/h(\lambda)$ is equal to the dimension of the irreducible
representation $\rho_\lambda$ of $\mathfrak S_n$ corresponding to the
partition $\lambda$. (See \cite[I.5.\ Ex.2 and I.(7.6)]{Mac}.)
\end{NB}

As a corollary of the above computation, we give a geometric proof of
the norm formula \cite[VI, (10.16)]{Mac}.

\begin{Proposition}
\begin{equation*}
  \la P^{(\bk)}_\lambda, P^{(\bk)}_\lambda\ra
  = \prod_{s\in\lambda} 
  \frac{l(s) + (a(s)+1) \bk}
  {l(s)+1 + a(s) \bk}.
\end{equation*}
\end{Proposition}

\begin{proof}
This is a direct consequence of
\begin{equation*}
  \la \frac1{e(T^{\le 0}_{I_\lambda})}[I_\lambda], 
  \frac1{e(T^{\le 0}_{I_\lambda})}[I_\lambda]\ra
  = (-1)^n \frac{e(T_{I_\lambda}\Hilb{n})}{e(T^{\le 0}_{I_\lambda})^2}
\end{equation*}
and the expression of $e(T^{\le 0}_{I_\lambda})$ in \corref{cor:Euler}.
\begin{NB}
  \begin{equation*}
  = \prod_{s\in\lambda} 
  \frac{\left(l(s)\ve_1 - (a(s)+1) \ve_2\right)}
  {\left((l(s)+1)\ve_1 - a(s) \ve_2\right)}
  \end{equation*}
\end{NB}%
\end{proof}

\begin{NB}
Let us denote $\la P^{(\bk)}_\lambda, P^{(\bk)}_\lambda\ra^{-1}$ by
$b_\lambda^{(\bk)}$ (see \cite[VI, (10.16)]{Mac}).
\end{NB}

\subsection{Nested Hilbert scheme and Pieri formula}

Let us first explain the compatibility between the convolution product
and the fixed point formula.

Suppose $M_1$, $M_2$ are smooth $T$-varieties and $Z\subset M_1\times
M_2$ is a nonsingular $T$-invariant subvariety. We further assume that
the second projection $p_1\colon Z\to M_1$ is proper, and hence the
convolution product
\begin{equation*}
  H^T_*(M_2)\to H^T_*(M_1); \bullet \mapsto p_{1*}(p_2^*(\bullet))
\end{equation*}
is well-defined. Let $p_1^T$, $p_2^T$ denote the restriction of the
first and second projections to the fixed point set $Z^T$.
respectively.

Let $i_1$, $i_2$, $i_Z$ denote the inclusions of fixed point sets
$M_1^T$, $M_2^T$, $Z^T$ to $M_1$, $M_2$, $Z$ respectively. Let us
denote the normal bundles by $N_1$, $N_2$, $N_Z$ respectively. We
understand that they are union of normal bundles of connected
components of $M_1^T$, $M_2^T$, $Z^T$. We do not introduce subscript
$\alpha$ unlike in \subsecref{sec:fixed}. From \thmref{thm:inverse} we
obtain
\begin{Lemma}
  The following equality holds in $\Hom(H^T_*(M_1^T)_\loc,
  H^T_*(M_2)_\loc)$.
  \begin{equation*}
    p_{1*} p_2^* i_{2*} \frac1{e(N_2)}
    = i_{1*} p_{1*}^T \frac1{e(N_Z)} p_2^{T*}.
  \end{equation*}
\end{Lemma}

In fact, we apply \thmref{thm:inverse} to $i_Z^* p_2^{*} = p_2^{T*}
i_2^*$ to invert $i_Z^*$, $i_1^*$. Then we use $p_{1*} i_{Z*} = i_{1*}
p_{1*}^T$. It is suggestive to note that
\begin{equation*}
  \frac{e(N_Z)}{p_2^{T*} e(N_2)}
\end{equation*}
is the equivariant Euler class of the virtual normal bundle 
$N_Z - p_2^{T*} N_2$ of fibers.

We apply this lemma to $P[1]\subset \Hilb{n}\times\Hilb{n-1}\times X$
in \subsecref{sec:heisenberg-algebra}, which realize the operator
$P_{-1}(\alpha)$.
This $P[1]$ is known to be nonsingular, while other $P[i]$ with $i>1$
are singular except for small $n$. The smoothness was proved in
\cite{Na-alg} in the context of quiver varieties, and also
independently in \cite{Cheah-nest,MR1609203} in the context of Hilbert
schemes. It is called the {\it Hecke correspondence\/}
\cite{Na-quiver,Na-alg}, and the {\it nested Hilbert scheme\/} in
\cite{Cheah-nest} and also various other literature. It has been used
to prove many statements on Hilbert schemes by an induction on $n$.
See \cite{MR1432198,MR1795551} for example.

Let us briefly review the proof of smoothness in
\cite[\S5]{Na-alg}. The proof works for higher rank case.
We represent $\Hilb{n}$ and $\Hilb{n-1}$ as spaces of quadruples
$(B_1,B_2,i,j)$ as in \cite[Ch.~2]{Lecture}: Let $W=\CC^r$, $V^1 =
\CC^n$, $V^2 = \CC^{n-1}$. Then the framed moduli spaces $M(r,n)$
($\alpha=1$), $M(r,n-1)$ ($\alpha=2$) of torsion free sheaves
$(E,\varphi)$ over $\proj^2$ of rank $r$, $c_2 = n$, $n-1$ are
respectively spaces of quadruples
$(B^\alpha_1,B^\alpha_2,i^\alpha,j^\alpha)$ satisfying
\begin{itemize}
\item $B^\alpha_1, B^\alpha_2\in \End(V^\alpha)$, $i^\alpha\colon W\to
  V^\alpha$, $j^\alpha\colon V^\alpha\to W$,
\item $[B^\alpha_1,B^\alpha_2]+i^\alpha j^\alpha = 0$,
\item (stability) there is no proper subspace of $V^\alpha$ containing
  $i^\alpha(W)$ and is invariant under $B^\alpha_1$, $B^\alpha_2$
\end{itemize}
modulo the conjugation under $\GL(V^\alpha)$. We consider
$V^\alpha$ as vector bundles over $M(r,n)$, $M(r,n-1)$.
We then form a complex of vector bundles over $M(r,n)\times
M(r,{n-1})$:
\begin{equation}\label{eq:8}
  \Hom(V^1,V^2)
  \overset{a}{\longrightarrow}
  \begin{matrix}
    \Hom(V^1,Q\otimes V^2) \\ \oplus\\
    \Hom(W,V^2) \\ \oplus \\
    \Hom(V^1,\Wedge^2 Q\otimes W)
  \end{matrix}
  \overset{b}{\longrightarrow}
  \begin{matrix}
      \Wedge^2 Q \otimes\Hom(V^1,V^2)\\\oplus \\\Wedge^2 Q\otimes
      \shfO
  \end{matrix},
\end{equation}
where $Q=\CC^2$, and $a$, $b$ are defined by
\begin{equation*}
  \begin{split}
  & a(\xi\oplus\lambda) =
    \left(\xi B_1^1 - B_1^2 \xi \right)\oplus
    \left(\xi B_2^1 - B_2^2 \xi \right)\oplus
    \xi i^1\oplus
    \left(- j^2 \xi\right),
\\
   & b(C_1 \oplus C_2 \oplus I \oplus J)
   =
   \begin{pmatrix}
   B^2_1 C_2 - C_2 B^1_1 + C_1 B^1_2 - B^2_2 C_1 + i^2 J + I j^1\\
   \tr (i^1 J) + \tr(I j^2)
   \end{pmatrix}.
  \end{split}
\end{equation*}
This is a complex thanks to $[B^\alpha_1,B^\alpha_2]+i^\alpha j^\alpha = 0$
and $\tr(i^1 j^2\xi) = \tr(\xi i^1j^2)$.

\begin{Lemma}
  Consider $a$ and $b$ as linear map between fibers of vector
  bundles. Then $a$ is injective and $b$ is surjective at any point in
  $M(r,n)\times M(r,n-1)$.
\end{Lemma}

\begin{proof}
  \begin{NB}
  Suppose that $\xi\in\Ker a$. Then the stability condition applied to
  $\Ker\xi$ implies $\Ker\xi = V^1$, i.e., $\xi = 0$.

  Suppose that $b$ is not surjective. Then we have
  $\zeta\oplus\lambda\in \Hom(V^2,V^1)\oplus\CC$ orthogonal to the
  image of $b$. Therefore
  \begin{equation*}
    \zeta B^2_1 = B^1_1 \zeta, \quad
    \zeta B^2_2 = B^1_2 \zeta, \quad
    \zeta i^2 + \lambda i^1 = 0, \quad
    j^1\zeta + \lambda j^2 = 0.
  \end{equation*}
  If $\lambda\neq 0$, the image of $\zeta$, from the stability
  condition for $\alpha=1$, must be the whole $V^1$. But this is not
  possible as $\dim V^2 < \dim V^1$. Therefore $\lambda = 0$. Then
  $\Ker\zeta$, from the stability condition for $\alpha=2$, must be
  $V^2$. Therefore $\zeta = 0$.
  \end{NB}%
  See \cite[Lemma~5.2]{Na-alg}. Note that $V^1$, $V^2$ are swapped and
  the stability condition is the opposite there. Therefore we need to
  take transposes of $B^\alpha_1,B^\alpha_2, i^\alpha, j^\alpha$.
\end{proof}

Therefore $\Ker b/\Ima a$ forms a vector bundle over $M(r,n)\times
M(r,n-1)$ of rank $2rn-r-1$.

If we omit the factor $\shfO$ in the third term in \eqref{eq:8}, the
complex gives $\Ext^1(E_1, E_2(-\linf))$, where $E_1$, $E_2$ are
torsion free sheaves in $M(r,n)$ and $M(r,n-1)$ respectively. This is
clear from the proof of the description for $M(r,n)$, $M(r,n-1)$. It
can be also checked as follows: we consider $M(r,n)\times M(r,n-1)$ as
a component of $\CC^*$-fixed point in $M(2r,2n-1)$, where $\CC^*$ acts
on $\CC^{2r} = \CC^r\oplus \CC^r$ with weight $0$ on the first factor
and $1$ on the second. The tangent space of $M(2r, 2n-1)$ at
$E_1\oplus E_2$ is $\Ext^1(E_1\oplus E_2, (E_1\oplus E_2)(-\linf))$,
and is a $\CC^*$-module. The factor $\Ext^1(E_1, E_2(-\linf))$ is the
weight $1$ subspace. On the other hand, we compute the tangent space
in terms of quadruples, and find that \eqref{eq:8} without $\shfO$ is
the weight $1$ subspace.

We define a section $s$ of $\Ker b/\Ima a$ by
\begin{equation*}
  s = \left(0\oplus i^2\oplus (-j^1)\right) \pmod{\Ima a}.
\end{equation*}
Then $s$ vanishes if and only if there is $\xi$ such that
\begin{equation*}
  \xi B^1_1 = B^2_1\xi, \quad \xi B^1_2 = B^2_2\xi, \quad\xi i^1 = i^2,
  \quad j^1 = j^2\xi.
\end{equation*}
This means we have a homomorphism $E_1\to E_2$ which is the identity
at $\linf$. It is injective, as a sheaf homomorphism, since it is so
at $\linf$.
For $r=1$, this means that $I_1\subset I_2$. The stability condition
implies $\xi$ is surjective, hence $\Ker\xi$ is $1$-dimensional. The
homomorphisms $B^1_1$, $B^2_2$ preserve $\Ker\xi$, hence give an
element in $\CC^2 = X$. Thus $\operatorname{Zero}(s) = P[1]$ if $r=1$,
and its natural higher rank generalization for $r>1$.

Now the smoothness of $\operatorname{Zero}(s)$ follows from
\begin{Lemma}
  Take a connection $\nabla$ on $\Ker b/\Ima a$ and consider the
  differential $\nabla s$. It is surjective on
  $\operatorname{Zero}(s)$.
\end{Lemma}

See \cite[Th.~5.7]{Na-alg} for the proof. Therefore
$\operatorname{Zero}(s)$ is nonsingular of dimension $2rn-r+1 =
\frac12 \dim (M(r,n)\times M(r,n-1)\times X)$.

Now assume $r=1$ and consider a $T^2$-action on $P[1]$ induced by
$(t_1,t_2)\cdot (z,\xi) = (t_1 z, t_2 \xi)$. The fixed points are
parametrized by pairs of Young diagrams $(\lambda,\mu)$ of size $n$
and $n-1$ respectively such that $\mu$ is contained in $\lambda$,
i.e., $\lambda_i\ge \mu_i$ for all $i\ge 1$. Since the number of boxes
differs by $1$, the skew diagram $\lambda - \mu$ consists of a single
box.

Using the complex \eqref{eq:8}, we can compute the character of the
tangent space at $(I_\lambda,I_\mu)$. We consider $V^1$, $V^2$ as
$T$-modules according to $\lambda$, $\mu$, and $Q$ as the natural
$T$-module.
The detail is given in \cite[Prop.~5.2]{perv2}. (Note that $\widehat
M^0(c)$ has no $\CC^2$-factor, so $t_1+t_2$ below is not present there.)

\begin{Proposition}
  \begin{equation*}
    \ch T_{(I_\mu,I_\lambda)} P[1] = t_1 + t_2 +
    \sum_{s\in\mu} \left(t_1^{-l_\lambda(s)} t_2^{a_\mu(s)+1} + 
    t_1^{l_\mu(s)+1} t_2^{-a_\lambda(s)}\right).
  \end{equation*}
\end{Proposition}
Here the leg and arm lengths are considered with respect to either
$\lambda$ or $\mu$. We use the notation $l_\lambda(s)$, $l_\mu(s)$ to
indicate Young diagrams as subscripts.

Let
\begin{equation}
  b^{(\bk)}_\lambda(s) \defeq
  \frac{\ve_1(l_\lambda(s)+1)-\ve_2 a_\lambda(s)}
  {\ve_1 l_\lambda(s) -\ve_2(a_\lambda(s)+1)}
  =
  \frac{l_\lambda(s)+1 + \bk a_\lambda(s)}
  {l_\lambda(s) + \bk (a_\lambda(s)+1)}.
\end{equation}
See \cite[(10.10)]{Mac}.

Let us denote $R$ the set of boxes in $\mu$ which lies in the same row
with the box $\lambda - \mu$. For example, when $\lambda - \mu$ is the
box marked with $\hs$, $R$ consists of boxes with $\cs$:
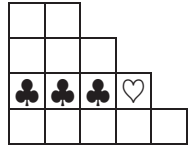
\begin{equation*}
\Yvcentermath1
\young(\hf\hf,\hf\hf\hf,\cs\cs\cs\hs,\hf\hf\hf\hf\hf)
\end{equation*}
In the notation \cite[VI, \S6]{Mac}, it is $C_{\lambda/\mu} -
R_{\lambda/\mu}$. (Note that our Young diagram is rotated by $90^\circ$.)

If $s\in \mu \setminus R$, we have $l_\lambda(s) = l_\mu(s)$. If $s\in
R$, we have $l_\lambda(s) = l_\mu(s) + 1$ and $a_\lambda(s) =
a_\mu(s)$. Therefore
\begin{equation*}
  \begin{split}
  & e(T_{(I_\lambda,I_\mu)} P[1])
\\
  =\; & \ve_2 e(T^{>0}_{I_\mu}) e(T^{\le 0}_{I_\lambda})
  \prod_{s\in R}
  \frac{-\ve_1 l_\lambda(s) + \ve_2 (a_\lambda(s)+1)}
  {-\ve_1 l_\mu(s) + \ve_2 (a_\mu(s)+1)}
  \frac{\ve_1 (l_\mu(s) + 1) - \ve_2 a_\mu(s)}
  {\ve_1 (l_\lambda(s) + 1) - \ve_2 a_\lambda(s)}
\\
  =\; &
  \ve_2 e(T^{>0}_{I_\mu}) e(T^{\le 0}_{I_\lambda})
  \prod_{s\in R}
  \frac{b^{(\bk)}_\mu(s)}{b^{(\bk)}_\lambda(s)},
  \end{split}
\end{equation*}
where $T^{>0}_{I_\mu}$, $T^{\le 0}_{I_\lambda}$ are sum of positive
and negative weight spaces as before.

Since $P^{(\bk)}_\mu = \frac1{e(T^{\le 0}_{I_\mu})}[I_\mu]$ and $p_1 =
P_{-1}(\ve_2)$, we get

\begin{Theorem}
We have
    \begin{equation*}
        p_1 P_\mu^{(\bk)}
        = \sum_\lambda   \prod_{s\in R}
        \frac{b^{(\bk)}_\lambda(s)}{b^{(\bk)}_\mu(s)} P_\lambda^{(\bk)},
    \end{equation*}
    where the summation runs over $\lambda$ with $|\lambda|=|\mu|+1$,
    containing $\mu$.
\end{Theorem}

This is a special case of Pieri formulas for Jack symmetric functions
\cite[VI, (6.24)]{Mac}. When $\bk=1$, we have $b^{(\bk)}_\lambda(s) =
1$. Then the above is specialized to a classical Pieri formula for
Schur functions.
\begin{NB}
Since
    \begin{equation*}
        \frac{e(N_Z)}{p_2^{T*}e(N_2)} = 
        \frac{e(T_{(I_\lambda,I_\mu)}P[1])}{e(T_{I_\mu})}
        = \ve_2 \frac{e(T_{I_\lambda}^{\le 0})}{e(T_{I_\mu}^{\le 0})}
        \prod_{s\in R} \frac{b^{(\bk)}_\mu(s)}{b^{(\bk)}_\lambda(s)},
    \end{equation*}
we get
    \begin{equation*}
        P_{-1}(\ve_2) \frac1{e(T^{\le 0}_{I_\mu})}[I_\mu]
        = \sum_\lambda \frac1{e(T^{\le 0}_{I_\lambda})}[I_\lambda]
         \prod_{s\in R}
        \frac{b^{(\bk)}_\lambda(s)}{b^{(\bk)}_\mu(s)}.
    \end{equation*}
\end{NB}

\section{Virasoro algebra}\label{sec:Vir}

The goal of this section is to construct a representation of the
Virasoro algebra on $\bigoplus H^T_*(\Hilb{n})$ for $X=\CC^2$. Lehn
\cite{Lehn} constructed a representation for an arbitrary
quasiprojective surface $X$. Our proof is completely different from
Lehn's, and based on the geometric construction of Jack symmetric
functions in the previous section. The key is that the Hamiltonian
$\square^{\bk}$ gives the Virasoro algebra. On the other hand, Chern
classes of the tautological bundle $\mathcal V$ on $\Hilb{n}$ is
diagonalized by the fixed points base $\{[I_\lambda]\}$. Therefore it
is, more or less, obvious that $\square^{\bk}$ is related to $\mathcal
V$. As far as the author knows, this observation was not mentioned
explicitly in the literature before, but it has been well-known among
experts.

\subsection{Insertions and coproducts}

Note that our Heisenberg generators $\PP{\alpha}{i}$ are `colored' by
(co)homology classes $\alpha$ of the base space $X$. When we consider
a (normal ordered) product of Heisenberg operators, such as a vertex
operator, it is sometime natural to color it a single cohomology class
instead of multiples of them. Such a coloring is given naturally by
considering a coproduct on $H^*_T(X)$. This was first noticed by Lehn
\cite{Lehn} when he considered Virasoro generators, and subsequently
used by other people. Here we use a slightly modified version in
\cite{MO}. We only consider the case $X = \CC^2$, but note that the
framework makes sense for any $X$ with or without the $T$-action.

Let $\Delta\colon H^*_T(X)\to H^*_T(X)\otimes H^*_T(X)$ be the adjoint
of the cup product $\cup\colon H^*_T(X)\otimes H^*_T(X)\to H^*_T(X)$
with respect to $\la\ ,\ \ra$ \eqref{eq:18}, the negative of the
intersection pairing. More concretely, it is
$H^*_T(\mathrm{pt})$-linear, and hence is given by
\begin{equation*}
  \Delta(1) = - 1\otimes \operatorname{PD}[0] = - 1\otimes \ve_1\ve_2,
\end{equation*}
where $\operatorname{PD}[0]$ is the Poincar\'e dual of the class
$[0]$. We also consider its iteration
\begin{equation*}
  \Delta^n(1) = (-1)^n \otimes \underbrace{
  \ve_1\ve_2\otimes\cdots\otimes\ve_1\ve_2}_{\text{$n$ times}}.
\end{equation*}
Note that we can iterate $\Delta$ in various ways, say $(\Delta\otimes
1)\Delta$, $(1\otimes\Delta)\Delta$, but the result is the same thanks
to the coassociativity. Thus we have $\Delta^n(1) = (-\ve_1\ve_2)^n\cdot 
1\otimes\cdots\otimes 1$.

We define
\begin{equation*}
  (P_{m} P_{n})(\alpha) \defeq \sum_i \PP{\alpha_i'}{m}\PP{\alpha_i''}{n}
\end{equation*}
if $\Delta(\alpha) = \sum \alpha_i'\otimes\alpha_i''$. Similarly we
define 
\begin{equation*}
  \normal{P_m P_n}(\alpha) \defeq \sum_i 
  \normal{\PP{\alpha_i'}{m} \PP{\alpha_i''}{n}}.
\end{equation*}
Products of more than two operators are defined in the same way, using
$\Delta^n(\alpha)$.

\begin{Remark}
  Our $\Delta$ is the negative of $\delta\colon H^*_T(X)\to
  H^*_T(X)\otimes H^*_T(X)$, the pushforward homomorphism associated
  with the diagonal embedding, used in \cite[\S3.1]{Lehn}.
\end{Remark}

\subsection{The first Chern class of the tautological bundle}

Let $\mathcal Z$ be the universal family on $\Hilb{n}$, which is a
subvariety of $X\times\Hilb{n}$. Let $p$ denote the projection to
$\Hilb{n}$. Then $p_*\shfO_{\mathcal Z}$ is a vector bundle of rank
$n$ over $\Hilb{n}$. In the description \cite[Th.~1.9]{Lecture}, it is
the vector bundle associated with the principal $GL_n(\CC)$-bundle
$\widetilde H \to \Hilb{n} = \widetilde H/GL_n(\CC)$ with respect to
the vector representation of $GL_n(\CC)$. Let us denote it by
$\mathcal V$. It is denoted by $\shfO^{[n]}$ in \cite{Lehn}. It is
called the {\it tautological bundle}.

We consider the multiplication of the first Chern class $c_1(\mathcal
V)$ of $\mathcal V$ as an operator on $H^*_T(\Hilb{n})$.

\begin{Lemma}
  In the fixed points base $\{ [I_\lambda]\}$, the operator
  $c_1(\mathcal V)\cup\bullet$ is diagonalized:
  \begin{equation}\label{eq:16}
    c_1(\mathcal V)\cup[I_\lambda] 
    = -\left(n(\lambda) \ve_1 + n(\lambda') \ve_2\right)[I_\lambda].
  \end{equation}
\end{Lemma}

\begin{proof}
  Since $[I_\lambda]$ is the class of a fixed point, $c_1(\mathcal
  V)\cup[I_\lambda]$ is given by $c_1(\mathcal
  V|_{I_\lambda})[I_\lambda]$. The formula of the character of
  $\mathcal V|_{I_\lambda}$ is given in the proof of
  \cite[Prop.~5.8]{Lecture}. Then the equivariant first Chern class is
\begin{equation*}
  c_1(\mathcal V|_{I_\lambda})
  = - \sum_{s\in\lambda} \left(l'(s) \ve_1 + a'(s) \ve_2\right)
  = - n(\lambda) \ve_1 - n(\lambda') \ve_2.
\end{equation*}
\end{proof}

The above lemma is easy to prove, but has the following remarkable
consequence.
\begin{Corollary}
  $c_1(\mathcal V)\cup\bullet$ is equal to $\ve_1 \square^{\bk}$ under
  the isomorphisms \textup{(\ref{eq:7}, \ref{eq:12})}, where
  $\square^{\bk}$ is given by \eqref{eq:limit}.
\end{Corollary}

\begin{proof}
  Recall that $\square^{\bk}$ is diagonalized in Jack symmetric
  functions $P^{(\bk)}_\lambda$ with eigenvalues $e_\lambda(\bk) =
  n(\lambda')\bk - n(\bk)$, see \eqref{eq:17}. Since
  $P^{(\bk)}_\lambda$ corresponds to the (normalized) fixed point
  class in \thmref{thm:Jack}, the assertion follows as $\ve_1
  e_\lambda(\bk) = c_1(\mathcal V|_{I_\lambda})$ in \eqref{eq:16}.
\end{proof}

Let us write $\square^{\bk}$ in terms of Heisenberg generators
$\PP{\alpha}{m}$. Recall our identification \propref{prop:ident} and the
commutation relation
\(
  \left[ \PP{\ve_2}{i}, \PP{\ve_2}{j}\right]
  = i\delta_{i+j,0} \bk \operatorname{id},
\)
we rewrite \eqref{eq:15} by
\begin{equation*}
  p_m\leftrightarrow \PP{\ve_2}{-m}, \qquad
  m \frac{\partial}{\partial p_m} \leftrightarrow \frac1{\bk}\PP{\ve_2}{m}
  = - \PP{\ve_1}{m}.
\end{equation*}
We get
\begin{equation*}
  \begin{split}
  & \ve_1 \square^{\bk}
\\
  =\; & - \sum_{m,n>0}
  \left(\frac{\ve_2}2 \PP{\ve_2}{-m-n} \PP{\ve_1}{m} \PP{\ve_1}{n}
  + \frac{\ve_1}2 \PP{\ve_2}{-m} \PP{\ve_2}{-n} \PP{\ve_1}{m+n}\right)
\\
  &\quad + \frac{\ve_1+\ve_2}2 \sum_{m>0} (m-1) \PP{\ve_2}{-m} \PP{\ve_1}{m}.
  \end{split}
\end{equation*}
Noticing $\Delta(1) = \ve_1\ve_2$, $\Delta^2(1) = (\ve_1\ve_2)^2$, and
$K_X = -\ve_1-\ve_2$, we rewrite this expression as
follows. ($\PP{\alpha}{0}$ is understood as $0$.) It is a special case of
Lehn's result which holds for any $X$.

\begin{Theorem}[Lehn \protect{\cite{Lehn}}]\label{thm:cubic}
  \begin{equation*}
    \begin{split}
    c_1(\mathcal V)\cup\bullet
    =  - \frac1{3!} \sum_{m_1+m_2+m_3=0} 
  & \normal{P_{m_1} P_{m_2} P_{m_3}}({1})
\\ 
  & + \frac{1}{4} \sum_{m} (|m|-1) \normal{P_{-m} P_{m}}({K_X}).
    \end{split}
  \end{equation*}
\end{Theorem}

The commutator with $\PP{\alpha}{n}$ is given by
\begin{equation*}
  \left[c_1(\mathcal V)\cup\bullet, \PP{\alpha}{n}\right]
  = 
    \frac{n}2 \sum_{l+m=n} \normal{P_l P_m}(\alpha)
    - \frac{n(|n|-1)}2 P_n({K_X\cup\alpha}).
\end{equation*}
The formula is presented in this form in \cite[Th.~3.10]{Lehn} up to
the sign convention.
Since $\bigoplus_n H^T_*(\Hilb{n})$ is an irreducible representation
of the Heisenberg algebra, two formulas are equivalent.

\begin{Remark}
  Note that $\mathfrak q_n(\alpha)$ in \cite{Lehn} is our
  $\PP{\alpha}{-n}$.
\end{Remark}

The terms
\begin{equation*}
  L_n(\alpha) \defeq \frac12  \sum_{l+m=n} \normal{P_l P_m}(\alpha)
\end{equation*}
satisfy the Virasoro relation
\begin{equation}\label{eq:19}
  \left[L_n(\alpha), L_m(\beta)\right] = (n-m) L_{n+m}(\alpha\beta)
  - \frac{n^3-n}{12} \delta_{n+m,0} \la c_2(X) \alpha,\beta\ra
  \operatorname{id},
\end{equation}
where $c_2(X) = \ve_1\ve_2$. It appears as the composite
\begin{equation*}
  H^*_T(X)\xrightarrow{\Delta} H^*_T(X)\otimes H^*_T(X)
  \xrightarrow{\cup} H^*_T(X)
\end{equation*}
is multiplication with $-c_2(X)$. We leave a check of \eqref{eq:19} to
the reader.

\begin{Remark}
  It would be nice if one could deduce Lehn's result for general $X$
  from \thmref{thm:cubic}. In fact, Li-Qin-Wang \cite{LQW-universal}
  (see also \cite{Lehn-lecture}) proved that operators given by cup
  products of classes
  \begin{equation*}
    p_{1*}(\ch(\shfO_{\mathcal Z_n}) p_2^*(\operatorname{td}_X\cup\alpha)),
  \end{equation*}
  are expressed by universal polynomials in Heisenberg operators
  coupled with $\alpha\cup \{ 1_X, K_X, K_X^2, c_2(X) \}$, independent
  of $X$. Here $\mathcal Z_n$ is the universal subscheme in
  $\Hilb{n}\times X$, $p_1$, $p_2$ are projection from $\Hilb{n}\times
  X$ to $\Hilb{n}$ and $X$ respectively, and $\operatorname{td}_X$ is
  the Todd class of $X$. By Riemann-Roch, $c_1(\mathcal V)$ is the
  degree $2$ part of the above for $\alpha=1_X$. Thanks to the
  universality, it is enough to determine the polynomials for
  $X=\CC^2$. However the universality in \cite{LQW-universal} was
  derived from Lehn's result, and hence this argument does not work.
\end{Remark}

\bibliographystyle{myamsplain}
\bibliography{nakajima,mybib,tensor2,refined}

\end{document}